\documentclass[11pt]{article}
\usepackage[latin1]{inputenc}
\usepackage{amsmath,amsthm,amssymb}
\usepackage{amsfonts}
\usepackage{amsmath,amsthm,amssymb,amscd}
\usepackage{latexsym}
\usepackage{color}
\usepackage{graphicx}
\usepackage{mathrsfs}

\makeatletter \@addtoreset{equation}{section} \makeatother

\newtheorem{theorem}{Theorem}[section]

\newtheorem{proposition}{Proposition}[section]
\newtheorem{lemma}{Lemma}[section]

\newtheorem{corollary}[theorem]{Corollary}

\begin{document}

\title{\bf Concentration Phenomena of Normalized Solutions \\ of Critical Biharmonic Equations \\ with Combined Nonlinearities
in $\mathbb{R}^{N}$}
\author{{\bf Yueqiang Song$^{a}$, Jiaying Ma$^{a}$,  Du\v{s}an D. Repov\v{s}$^{b,c,d}$
            \thanks{{{{\it E-mail address:} songyq16@mails.jlu.edu.cn (Y. Song), 15674240256@163.com (J. Ma), dusan.repovs@guest.arnes.si (D.D. Repov\v{s}).}}}
            \thanks{Corresponding author: Du\v{s}an D. Repov\v{s}}}\\
            $^{\small\mbox{a}}${\small  College of Mathematics, Changchun Normal
            University,   Changchun, 130032,  P.R. China}\\[-0.2cm]
        $^{\small\mbox{b}}${\small  Faculty of Education, University of Ljubljana, Ljubljana, 1000, Slovenia }\\[-0.2cm]
        $^{\small\mbox{c}}${\small  Faculty of Mathematics and Physics, University of Ljubljana, Ljubljana, 1000, Slovenia}\\[-0.2cm]
        $^{\small\mbox{d}}${\small  Institute of Mathematics, Physics and Mechanics, University of Ljubljana, Ljubljana, 1000, Slovenia}}
\date{}
\maketitle

\begin{abstract}
We  prove the multiplicity and concentration of normalized solutions of
critical biharmonic equations with combined nonlinearities in $\mathbb{R}^{N}$
\begin{equation*}
\Delta^{2}u+V(\varepsilon x)u=\lambda u+\mu |u|^{q-2}u+|u|^{2^{**}-2}u
\mbox{ in }\ \mathbb{R}^{N}, \quad    \int_{\mathbb{R}^{N}}|u|^{2}dx=c^{2},
\end{equation*}
where $\Delta^{2}$ is the biharmonic operator, $N\geq5$, $\mu,c>0$, $\varepsilon>0,$
$\lambda\in\mathbb{R}$, $q\in(2,2+\frac{8}{N}),$ and $2^{**}=\frac{2N}{N-4}$ is the Sobolev
critical exponent. The potential $V$ is a bounded and continuous nonnegative
function, satisfying some suitable global conditions. Using minimization
techniques and a truncation argument, we show that the number of normalized solutions is
not less than the number of global minimum points of $V$ when the parameter
$\varepsilon$ is sufficiently small. To overcome the loss of
compactness of the energy functional due to the critical growth, we apply the concentration-compactness principle. To the best of our knowledge, this study is  the first contribution regarding the concentration and multiplicity properties of normalized
solutions of critical biharmonic equations with combined nonlinearities in $\mathbb{R}^{N}$. To some extent,  the
main results included in this paper complement several recent contributions to the study of biharmonic equations with combined nonlinearities.
\end{abstract}

\maketitle

\section{Introduction}\label{S1}

In this paper, we intend to study the multiplicity and concentration
of normalized solutions of
 critical biharmonic equations with combined nonlinearities in $\mathbb{R}^{N}$ of the form
\begin{equation}\label{e1.1}
\Delta^{2}u+V(\varepsilon x)u=\lambda u+\mu |u|^{q-2}u+|u|^{2^{**}-2}u
\mbox{ in }\ \mathbb{R}^{N},\ \int_{\mathbb{R}^{N}}|u|^{2}dx=c^{2},
\end{equation}
where $\Delta^{2}$ is the biharmonic operator, $N\geq5$, $\mu,c>0$, $\lambda\in\mathbb{R}$, {$q\in(2,2+\frac{8}{N})$},
and
$2^{**}=\frac{2N}{N-4}$ is the Sobolev critical exponent.
Throughout the paper, we will assume that the potential function $V$ is a bounded and nonnegative continuous function,
   satisfying the following conditions:
\begin{itemize}
\item[$(V_{1})$] $V\in L^{\infty}(\mathbb{R}^{N})$, $V(x)\geq0,$ for every $x\in\mathbb{R}^{N}$.
\item[$(V_{2})$]  $V_{\infty}=\lim_{|x|\rightarrow+\infty}V(x)>V_{0}:=\min_{x\in\mathbb{R}^{N}}V(x)=0$.
\item[$(V_{3})$] $V^{-1}(\{0\})=\{k_{1},k_{2},k_{3},\cdots,k_{l}\}$, where $k_{1}=0$
  and $k_{j}\neq k_{l}$,  if $j\neq l$.
\end{itemize}

Over the past few decades, the biharmonic equation and
its higher-order elliptic equations have
 received a
 great deal of attention, due to their applications in physics and conformal geometry, e.g.
the biharmonic equation can be used to describe the problems of nonlinear oscillation in a suspension bridge,
see Lazer and McKenna \cite{la} and
 McKenna and Walter \cite{mc1},
 and the problem of the static deflections of an elastic plate
in a fluid, see Abrahams and Davis \cite{ab}. In general, there are two aspects to the study of system \eqref{e1.1}.

For the first case, i.e., for the fixed frequency $\lambda$, 
the objective is to find the critical points of the variational functional
$\Upsilon_{\lambda}:H^{2}(\mathbb{R}^{N})\rightarrow\mathbb{R}$
defined by
\begin{equation*}
\Upsilon_{\lambda}(u)=\frac{1}{2}\int_{\mathbb{R}^{N}}(|\Delta u|^{2}+V(\varepsilon x)|u|^{2}-\lambda|u|^{2})
dx-\frac{\mu}{q}\int_{\mathbb{R}^{N}}|u|^{q}dx-\frac{1}{2^{**}}\int_{\mathbb{R}^{N}}|u|^{2^{**}}dx.
\end{equation*}
In recent years, there have been many papers concerning the existence
and multiplicity of solutions for this case.
Some authors have obtained
nontrivial solutions to semilinear biharmonic problems involving critical exponents by using variational methods, see  Alves and  do \'{O} \cite{al2} and Alves et al. \cite{al1}.
Liang et al. \cite{liang1} studied the existence and multiplicity of solutions of biharmonic
equations with critical nonlinearity. Deng and Shuai \cite{de1} proved the existence of
nontrivial solutions for a class of semilinear biharmonic
problems with critical growth and potential
vanishing at infinity. Carranza and Pimenta \cite{car}  studied some systems of elliptic PDEs involving the $\mathit{1}$-Laplacian operator, where the strategy is based on approximation arguments to conclude the existence of solutions as the limit of related problems with the $p$-Laplacian operator. The authors also used a version of  Lions' concentration of compactness principle and suitable estimates. Hai and Zhang \cite{hai} obtained the existence results for nonhomogeneous Choquard equation involving $p$-biharmonic operator and critical growth by using the concentration-compactness principle together with the mountain pass theorem.  Some other interesting results on this topic can be found in  \cite{dw, liang2, sa, zh2, zh1}.

On the other hand, from a physical perspective, there is interest in finding solutions of
 system \eqref{e1.1} with prescribed
mass. For this case, the parameter $\lambda\in\mathbb{R}$ can be regarded as a Lagrange
multiplier, which is determined by the solution and is not given
a priori.  To the best of our knowledge, the study of $L^{2}$-constrained problems can give
a better insight into
 the dynamical properties, since the variational
characterization of these solutions is often very helpful for analyzing
their orbital stability, see Jeanjean et al. \cite{jl2},
Jeanjean and Le \cite{jl1},
and
Li \cite{lx}.
The aim of the present paper is to establish the existence
of multiple weak solutions of
 system \eqref{e1.1}. Throughout the
paper, a solution always refers to a couple $(u,\lambda)$
that satisfies system \eqref{e1.1}. We call these solutions  normalized
solutions, since prescribed mass imposes a normalization on the $L^{2}$-norm of $u$.

Very recently, some  authors have studied the existence,  multiplicity and other properties
of normalized solutions under some assumptions by variational methods.
If $V(x)=0$ and without the critical exponent in system \eqref{e1.1},
Luo and Yang \cite{ly} studied the following nonlinear biharmonic
Schr\"{o}dinger equations of the form
\begin{equation*}
\Delta^{2}u+\mu\Delta u-\lambda u=|u|^{p-2}u
\mbox{ in }\ \mathbb{R}^{N},\
\int_{\mathbb{R}^{N}}|u|^{2}=a^{2},
\end{equation*}
where $N\geq5$, $a,\mu>0$, $2+\frac{8}{N}<p<4^{*}=\frac{2N}{N-4},$
and $\lambda\in\mathbb{R}$. They established
some asymptotic properties of the normalized solutions, as $\mu\rightarrow0^{+}$
and $a\rightarrow0^{+}$.

Wang et al. \cite{ww} considered the following nonlinear biharmonic Schr\"{o}dinger equation:
\begin{equation*}
\Delta^{2}u = \lambda u + h(\varepsilon x)|u|^{p-2}u
\mbox{ in }\ \mathbb{R}^{N},\
\int_{\mathbb{R}^{N}}|u|^{2}=c^{2},
\end{equation*}
where $c, \varepsilon > 0$,  $N \geq 5$, $\lambda \in \mathbb{R}$ is a Lagrange multiplier and is unknown, $h\in C(\mathbb{R}^N, [0, +\infty))$,
and
$f: \mathbb{R} \rightarrow \mathbb{R}$
is a continuous function satisfying $L^2$-subcritical growth. For sufficiently small
 $\varepsilon>0,$
  they proved the existence of multiple normalized
solutions.
Moreover, they also established
 orbital stability of the solutions of
 this problem.
For more results on nonlinear 
 elliptic equations, we refer
readers to \cite{cc,  cw2, liang0, tong1, xiao, zg}
 and  the references
therein.

For critical case, Liu and Zhang \cite{lz} investigated the following biharmonic nonlinear Schr\"{o}dinger equation
  with prescribed $L^2$-norm:
\begin{equation*}\
\Delta^{2}u-\lambda u=\mu |u|^{q-2}u+|u|^{4^{*}-2}u
\mbox{ in }\ \mathbb{R}^{N},\
u\in H^{2}(\mathbb{R}^{N}),~\int_{\mathbb{R}^{N}}|u|^{2}dx=c>0 ,
\end{equation*}
where $N\geq5$, $ c>0$,
and $2+\frac{8}{N} < q < 4^{*}:=\frac{2N}{N-4}$.  They
proved  the existence of normalized solutions for  large enough $\mu>0,$ by verifying the $(PS)$ condition
at the corresponding mountain-pass level. In this sense, they extended the recent results obtained
by Ma and Chang \cite{mc} to the $L^2$-supercritical perturbation.

Chen and Chen \cite{cc}
considered the following  biharmonic Choquard equation with the
Hardy-Littlewood-Sobolev upper critical and combined nonlinearities:
\begin{equation}\label{e1.2}
\Delta^{2}u=\lambda u+\mu |u|^{q-2}u+(I_{\alpha}*|u|^{4_{\alpha}^{*}})|u|^{4_{\alpha}^{*}-2}u
\mbox{ in }\ \mathbb{R}^{N},\
\int_{\mathbb{R}^{N}}|u|^{2}dx=a>0,
\end{equation}
where $N\geq5$, $2<q<2+\frac{8}{N}$, $\mu>0$, $\alpha\in(0, N)$,  $\lambda \in  \mathbb{R}$ appears as a Lagrange multiplier,  $I_{\alpha}$ is
the Riesz potential, and $4_{\alpha}^{*}=\frac{N+\alpha}{N-4}$. Under appropriate assumptions,
they also obtained 
multiple normalized solutions of problem
\eqref{e1.2}.

In the present paper, we will be particularly interested in the existence of normalized solutions
of
the  equation with potential functions.
Namely, the methods in the papers cited above,
strongly depend on the fact that the potential functions $V$ are constant and their proofs
do not work for non-constant $V,$ even if $V$ is radial, i.e., $V(x) = V(|x|)$.
In recent years, some authors have begun to focus on the mass prescribed problem with potential
\begin{align}\label{a1.8}
-\Delta u+(V(x)+\lambda)u = h(u)  &\mbox{ in }\  \mathbb{R}^{N}, \
\displaystyle\int_{\mathbb{R}^{N}}|u|^{2}dx=a^{2}.
\end{align}

Ikoma and Miyamoto \cite{Ik}
used  the standard concentration compactness arguments as
in the seminal papers
by Lions \cite{P.Lion1, P.Lion2} to
obtain
the existence of a normalized solution of problem \eqref{a1.8} with
$V \leq 0$, $V(x)\rightarrow0,$ as $|x|\rightarrow\infty$
{and}  $h$ is subcritical mass.
The conditions on the potential function $V$
 have been considerably relaxed by
 Alves and Ji \cite{Aj} and
 Yang et al. \cite{Y2}.
In the subcritical mass case, one can try to minimize
{the underlying functional} $\mathcal{E}_\lambda$ on $S(a)=\{u \in
H^1(\mathbb{R}^N): |u|_2=a\}$.
 On the other hand,  so far
 for the mass supercritical $h$,
 only
 the  case of
a homogeneous nonlinearity $h(|u|)u = |u|^{p-2}u$ has been considered. Bartsch et al. \cite{Bar2} investigated
decaying potentials, i.e., $V(x)\rightarrow 0,$ as $|x|\rightarrow \infty,$
whereas  Bellazzini
et al. \cite{Bel} treated partially confining potentials.

It is  important to point out that Alves and Thin
\cite{AT} studied the following nonlinear Schr\"{o}dinger equation:
\begin{equation}\label{a1.9}
- \Delta u+V(\varepsilon x) u=\lambda u+g(u)  \mbox{ in }\  \mathbb{R}^{N}, \
\int_{\mathbb{R}^{N}}|u|^{2}dx=a^{2},
\end{equation}
where $a, \varepsilon>0$,  $g$ is a continuous function with $L^2$-subcritical growth, and
$V$ is a continuous function sa\-tisfy\-ing suitable conditions.
 With the help of the Lusternik-Schnirelmann category and
the penalization method, they proved the existence of multiple normalized solutions
of problem~\eqref{a1.9}.

For biharmonic equation with potential functions, Bellazzini and Visciglia \cite{bv1} considered the
following $L^{2}$-subcritical problem:
\begin{equation*}
\Delta^{2}u+V(x)u-\mathcal{Q}(x)|u|^{p-2}u=\lambda u
\mbox{ in }\ \mathbb{R}^{N},\
\int_{\mathbb{R}^{N}}u^{2}dx=c,
\end{equation*}
where $2<p<2+\frac{8}{N}$, $V$
and
$\mathcal{Q}\in L^{\infty}(\mathbb{R}^{N})$.
Under certain  conditions, they proved the existence of ground state
solutions. Moreover,  the orbital stability of the minimizers was also established.

Phan \cite{pt} considered the following biharmonic Schr\"{o}dinger equation with   $L^{2}$-critical nonlinearity:
\begin{equation}\label{e1.3}
\Delta^{2}u+V(x)u-a|u|^{\frac{8}{N}}u=\lambda u
\mbox{ in }\ \mathbb{R}^{N},\
\int_{\mathbb{R}^{N}}|u|^{2}dx=1,
\end{equation}
where parameter $a>0$ stands for the strength of the attraction. Under suitable conditions for the potential function
$V$, he proved that problem \eqref{e1.3} has at least one ground state solution if the parameter $a$
belongs to some specific interval.

Nevertheless, once we turn our attention to the critical biharmonic equation with combined nonlinearities in $\mathbb{R}^{N}$, we observe that the
literature is very scarce. Motivated by the previously mentioned papers, in
the present paper, we intend to prove the multiplicity and concentration
of normalized solutions of system \eqref{e1.1} with critical term.
To the best of our knowledge, there are no known results on the existence
of multiple normalized  solutions of system \eqref{e1.1}.
We now state the main result of the present paper.
\begin{theorem}\label{the1.1}
Suppose that conditions $(V_{1})-(V_{3})$ hold. Then, there exist $\tilde{\varepsilon}$,
$V_{*}$, and $\bar{c}$ such that system \eqref{e1.1} admits at least $l$
pairs of weak solutions $(u_{\varepsilon}^{i},\lambda_{\varepsilon}^{i})\in H^{2}
(\mathbb{R}^{N})\times\mathbb{R}$,  for $|V|_{\infty}<V_{*}$,
$\varepsilon\in(0,\tilde{\varepsilon}),$ and $c\in(0,\bar{c}]$,  with
$\int_{\mathbb{R}^{N}}|u_{\varepsilon}^{i}|^{2}dx=c$, $\lambda_{\varepsilon}^{i}<0$, 
for every $i \in \{1,2,\cdots,k\}$. Furthermore, each $u_{\varepsilon}^{i}$ has a
maximum point $v_{\varepsilon}^{i}\in\mathbb{R}^{N}$ such that 
$$
V(v_{\varepsilon}^{i})
\rightarrow V(x^{i})=V_{0},\quad
 \mbox{as }\, \varepsilon\rightarrow0^{+}.
 $$
\end{theorem}

The proof of Theorem \ref{the1.1} is based on suitable variational and topological arguments. Because of the appearance of the biharmonic
operator and critical exponent, we have to establish new estimates.  Moreover, due to $\lambda$
not being prescribed, the sequences of
approximated Lagrange multipliers have to be controlled. All these difficulties will make the problems we study very interesting but complex. In order to clarify our contributions in relation to
previous results, we emphasize the following points.
\begin{itemize}
\item[$(1)$]
The proof of the existence of normalized solutions is fairly delicate due to 
the presence of the  critical exponent,
combined with a non-constant potential $V(x)$.
When dealing with the $L^p$-subcritical case, we cannot directly prove
that
the functional $\mathcal{E}_\lambda$ is bounded from below  as in
 Alves and Ji \cite{Aj}
 and
Jeanjean  \cite{jea}.
 Moreover,
  the potential $V$ can have infimum equal to zero, and then the penalized method found in del Pino and Felmer \cite{del}
does not work well in our case.

\item[$(2)$]  The main difficulty in the present paper is the analysis of the convergence of constrained Palais-Smale sequences. Indeed, the critical term and the unbounded region occur at the same time  in system \eqref{e1.1} making the bounded $(PS)$ sequences not necessarily convergent. Hence, we have to consider
how the interaction between the nonlocal term and the nonlinear term will affect the existence
of solutions of system \eqref{e1.1}. Another  main difficulty  is that sequences of approximated
Lagrange multipliers have to be controlled, since $\lambda$ is not prescribed.  Furthermore, weak limits
of the Palais-Smale sequences could leave a constraint.
These facts produce  lack of compactness which we overcome,
using the concentration-compactness principles due to Lions \cite{P.Lion1, P.Lion2}.

\item[$(3)$] We apply the minimization techniques and the Lusternik-Schnirelmann
category to prove  the relationship between the potential and multiplicity and concentration of
normalized solutions. We think that the methods in our paper
can be applied to study the multiple normalized solutions for other kinds of operators, as well as autonomous problem.
\end{itemize}

This paper is organized as follows:
 In Section \ref{S2}, we introduce the variational
setting and give preliminary lemmas.
In Section \ref{S3}, we study the autonomous problem
with truncated function and show some essential properties of the autonomous problem corresponding energy functional. In Section \ref{S4}, we consider the energy functional of  the non-autonomous problem. By Moser iteration, we obtain that critical points of the truncated functional are actually the solution of the original system. In Section \ref{S5}, we establish the multiplicity of solutions of system
\eqref{e1.1} and then prove Theorem \ref{the1.1}. Finally, in Section \ref{S6},  we summarize the most important features of the main result.

\section{Preliminaries}\label{S2}

In this section, we will introduce the key definitions and results. For all other fundamental material used in this paper,
     we refer the reader to the comprehensive monograph by Papageorgiou et al.  \cite{PRR}.

     We begin with the following two definitions:

(1) $S_c=\{u\in H^{2}(\mathbb{R}^{N}):|u|_{2}=c\}$ is the sphere of
  radius $c>0$ defined with the norm $|\cdot|_{2}.$ \\

(2) $\Upsilon:H^{2}(\mathbb{R}^{N})\rightarrow\mathbb{R}$ with
  \begin{equation*}
  \Upsilon(u)=\frac{1}{2}\int_{\mathbb{R}^{N}}(|\Delta u|^{2}+V(\varepsilon x)|u|^{2})dx
  -\int_{\mathbb{R}^{N}}G(u)dx,
  \end{equation*}
  where
  $$G(t)=\frac{\mu}{q}|u|^{q}+\frac{1}{2^{**}}|u|^{2^{**}},\quad t\in\mathbb{R}.$$
  In this section, we denote
   the function $g(t)=\mu|t|^{q-2}t+|t|^{2^{**}-2}t$
  with $t\in\mathbb{R}$, and hence
   $G(t)=\int_0^t g(\tau)d\tau$.
Let $H^{2}(\mathbb{R}^{N})$ is the Sobolev space given by
  $$H^{2}(\mathbb{R}^{N})=\{u\in L^{2}(\mathbb{R}^{N}):|u|_{2}<\infty\},$$
  endowed with the following norm:
  $$\|u\|^{2}_{H^{2}}=\int_{\mathbb{R}^{N}}(|\Delta u|^{2}+|u|^{2})dx,$$
  which is equivalent to the following norm:
  $$\|u\|_{V_{0}}=(|\Delta u|_{2}^{2}+|u|_{2,V_{0}}^{2})^{\frac{1}{2}},\quad
  \|u\|_{2,V_{0}}^{2}=\int_{\mathbb{R}^{N}}V_{0}|u(x)|^{2}dx.$$
Set $H_{\varepsilon}(\mathbb{R}^{N})$ to be the completion of
  $C_{0}^{\infty}(\mathbb{R}^{N})$, with respect to the norm
  $$\|u\|_{\varepsilon}=(|\Delta u|_{2}^{2}+|u|_{2,V,\varepsilon}^{2})^{\frac{1}{2}},
  \quad \|u\|_{2,V,\varepsilon}^{2}=\int_{\mathbb{R}^{N}}V(\varepsilon x)|u(x)|^{2}dx.$$

We will need to make frequent use of the well-known Gagliardo-Nirenberg inequality.
\begin{lemma}\label{lem2.1}
(see Nirenberg \cite[Theorem in Lecture II]{nb}) Let $u\in H^{2}(\mathbb{R}^{N})$ and
suppose that there exists a constant $C_{N,m}$ depending on
$m,N$ such that
$$\|u\|_{m}\leq C_{N,m}\|u\|_{2}^{1-\gamma_{m}}\|\Delta u\|_{2}^{\gamma_{m}},$$
where $\gamma_{m}=\frac{N}{2}(\frac{1}{2}-\frac{1}{m})$.
In particular,
 when $m=2^{**}$,
we get that $\gamma_{2^{**}}=1$ and
\begin{equation}\label{e2.1}
S:=\inf_{u\in H^{2}(\mathbb{R}^{N})\setminus\{0\}}
\frac{\int_{\mathbb{R}^{N}}|\Delta u|^{2}dx}{\left(\int_{\mathbb{R}^{N}}|u|^{2^{**}}dx\right)^{\frac{2}{2^{**}}}}.
\end{equation}
\end{lemma}

In order to prove that the corresponding energy functional of system \eqref{e1.1} satisfies the
compactness condition, we will need the following concentration-compactness principle which is similar to the methods used in Lions \cite[Lemma I.1]{P.Lion1, P.Lion2} and Smets \cite[Lemma 2.1]{sm}.
Therefore, we will skip the proofs of Lemmas \ref{lem2.2} and \ref{lem2.3}.
\begin{lemma}\label{lem2.2}
Let $\{u_{n}\}$ be a sequence weakly converging  to $u$ in $H^{2}(\mathbb{R}^{N})$
such that 
$$|u_{n}|^{2^{**}}\rightarrow \nu\quad
\mbox{and} \quad|\Delta u_{n}|^{2}\rightharpoonup \kappa$$
in the sense of measures. Then, for some at most countable index set $I$,
\begin{itemize}
\item[$(i)$] $\nu=|u|^{2^{**}}+\sum_{i\in I}\delta_{x_{i}}v_{i}$, $v_{i}>0$;

\item[$(ii)$] $\kappa>|\Delta u_{n}|^{2}+\sum_{i\in I}\delta_{x_{i}}\mu_{i}$, $\mu_{x_{i}}>0$;

\item[$(iii)$] $\kappa_{i}\geq S\nu_{i}^{\frac{2}{2^{**}}}$.
\end{itemize}
Here, $x_{i}\in\mathbb{R}^{N}$, $\delta_{x_{i}}$ is the Dirac measure
at $x_{i},$
$S$ is given by \eqref{e2.1} and $\kappa_{i}$, $\nu_{i}$ are
some positive constants.
\end{lemma}
\begin{lemma}\label{lem2.3}
Let $\{u_{n}\}$ be a  sequence weakly converging to $u$
in $H^{2}(\mathbb{R}^{N})$ and define
\begin{itemize}
\item[$(i)$] $\nu_{\infty}=\lim_{R\rightarrow\infty}\limsup_{n\rightarrow\infty}\int
_{|x|>R}|u_{n}|^{2^{**}}dx;$

\item[$(ii)$] $\kappa_{\infty}=\lim_{R\rightarrow\infty}\limsup_{n\rightarrow\infty}\int
_{|x|>R}|\Delta u_{n}|^{2}dx.$
\end{itemize}
Then the quantities $\nu_{\infty}$ and $\kappa_{\infty}$ exist and satisfy
\begin{itemize}
\item[$(iii)$] $\limsup_{n\rightarrow\infty}\int_{\mathbb{R}^{N}}|u_{n}|^{2^{**}}dx=
\int_{\mathbb{R}^{N}}d\nu+\nu_{\infty}$;

\item[$(iv)$]  $\limsup_{n\rightarrow\infty}\int_{\mathbb{R}^{N}}|\Delta u_{n}|^{2}dx
=\int_{\mathbb{R}^{N}}d\kappa+\kappa_{\infty}$;

\item[$(v)$] $\kappa_{\infty}\geq S\nu_{\infty}^{\frac{2}{2^{**}}}$.
\end{itemize}
\end{lemma}
Given $c>0$, we study the following function:
$$F(c,m)=\frac{1}{2}-\frac{\mu}{q}C_{N,q}^{q}c^{\frac{4q-(q-2)N}{8}}m^{\frac{(q-2)N-8}{4}}
-\frac{1}{2^{**}}S^{-\frac{2^{**}}{2}}m^{2^{**}-2},\quad m>0.$$
In addition, if $c\in(0,\infty)$ is fixed, we assume that $F_{c}(m):=F(c,m)$.
Note that $m\gamma_{m}<2$ for $2<m<\frac{8+2N}{N}$, $m\gamma_{m}=2$ for
$m=\frac{8+2N}{N}$ and $m\gamma_{m}>2$ for $\frac{8+2N}{N}<m<\frac{2N}{N-4}$.
Since $2<q<2+\frac{8}{N}$,
similarly as in Jeanjean \cite[Lemma 2.1]{jl2},
we can see that $F_{c}(m)\rightarrow-\infty,$
as $m\rightarrow0^{+}$ and $F_{c}(m)\rightarrow-\infty,$ as $m\rightarrow\infty$.
Then there exists $m_{c}>0$, the function $F_{c}(m)$ has a unique global maximum,
and the maximum value satisfies
\begin{equation*}
\max_{m>0}F_{c}(m)
\left\{
\begin{array}{lll}
>0 &\mbox{if}\ c<m_{c},\\
=0 &\mbox{if}\ c=m_{c},\\
<0 &\mbox{if}\ c>m_{c}.
\end{array}
\right.
\end{equation*}
Moreover, for each $c>0$, we define
$$H(c,m)=\frac{1}{2}m^{2}-\frac{\mu}{q}C_{N,q}^{q}c^{\frac{4q-(q-2)N}{8}}m^{\frac{(q-2)N}{4}}-
\frac{1}{2^{**}}S^{-\frac{2^{**}}{2}}m^{2^{**}},\quad m>0.$$
Now, we study the properties of $H_{c}(m):=H(c,m)=m^{2}F_{c}(m)$.
For $c<m_{c}$, 
$$\lim\limits_{m\rightarrow0^{+}}H_{c}(m)=0^{-}\quad \mbox{and}\quad
\lim\limits_{t\rightarrow+\infty}H_{c}(t)=-\infty. $$
By \eqref{e2.1},
$H_{c}(m)$ attains its positive global maximum if $c<m_{c}$. In the sequel,
we will always assume that $a\leq\bar{c}<m_{c}.$ Thus, there will exist $0<R_{0}<R_{1}<+\infty$
depending on $c$ such that
$$H_{c}(R_{0})=H_{c}(R_{1})=0.$$
Moreover, $H_{c}(m)<0$ on the intervals $(0,R_{0})$ and
$(R_{1},+\infty),$ and $H_{c}(m)>0$ on the interval $(R_{0},R_{1}).$
Let $\eta:R^{+}\rightarrow[0,1]$ be a non-increasing and
$C^{\infty}$ function satisfying
\begin{equation*}
\eta(x)=
\left\{
\begin{array}{lll}
1 &\mbox{if}\ 0\leq x\leq R_{0},\\
0 &\mbox{if}\ x\geq R_{1}.
\end{array}
\right.
\end{equation*}
By Lemma \ref{lem2.1} and $(V_{1}),$ we have, for every $u\in S_{c}$,
\begin{eqnarray*}
\Upsilon_{\varepsilon}(u) 
&\geq&\nonumber
\frac{1}{2}\|\Delta u\|_{2}^{2}-\frac{\mu}{q}\int_{\mathbb{R}^{N}}|u|^{q}dx-\frac{\eta(\|\Delta u\|_{2})}{2^{**}}\int_{\mathbb{R}^{N}}|u|^{2^{**}}dx \\
&\geq&\frac{1}{2}\|\Delta u\|_{2}^{2}-\frac{\mu}{q}C_{N,q}c^{\frac{4q-(q-2)N}{8}}\|\Delta u\|_{2}^{\frac{(q-2)N}{4}}-
\frac{\eta(\|\Delta u\|_{2})}{2^{**}}S^{-\frac{2^{**}}{2}}\|\Delta u\|_{2}^{2^{**}}\\
&=& H(c,\|\Delta u\|_{2}).
\end{eqnarray*}
We study the truncated functional
$$\Upsilon_{\varepsilon,T}(u)=\frac{1}{2}\|\Delta u\|_{2}^{2}+\frac{1}{2}\int_{\mathbb{R}^{N}}V(\varepsilon x)|u|^{2}dx-\frac{\mu}{q}\int_{\mathbb{R}^{N}}|u|^{q}dx-\frac{\eta(\|\Delta u\|_{2})}{2^{**}}\int_{\mathbb{R}^{N}}|u|^{2^{**}}dx.$$
By Lemma \ref{lem2.1} and $(V_{1})$, we obtain
\begin{eqnarray*}
\Upsilon_{\varepsilon,T}(u) &\geq&\nonumber
\frac{1}{2}\|\Delta u\|_{2}^{2}-\frac{\mu}{q}C_{N,q}c^{\frac{4q-(q-2)N}{8}}\|\Delta u\|_{2}^{\frac{(q-2)N}{4}}-
\frac{\eta(\|\Delta u\|_{2})}{2^{**}}S^{-\frac{2^{**}}{2}}\|\Delta u\|_{2}^{2^{**}}\\
&=&H_{T}(c,\|\Delta u\|_{2}),
\end{eqnarray*}
where $$H_{T}(c,m):=\frac{1}{2}m^{2}-\frac{\mu}{q}C_{N,q}^{q}c^{\frac{4q-(q-2)N}{8}}m_{2}^{\frac{(q-2)N}{4}}-
\frac{\eta(m)}{2^{**}}S^{-\frac{2^{**}}{2}}m_{2}^{2^{**}},  \quad m>0.$$
It is easy to see that $H_{T,c}(m):=H_{T}(c,m)$ has the following properties:
\begin{equation*}
\left\{
\begin{array}{lll}
H_{T,c}(m)\equiv H_{c},~ \mbox{for\ every}\ m\in(0,R_{0}],\\
H_{T,c}(m)~ \mbox{is\ positive\ and\ strictly\ increasing\ on}\ (R_{0},+\infty).
\end{array}
\right.
\end{equation*}
So, we can choose sufficiently small $R_{0}>0$ such that
\begin{equation}\label{e2.3}
\frac{1}{2}m_{1}^{2}-\frac{1}{2^{**}}S^{-\frac{2^{**}}{2}}m_{1}^{2^{**}}\geq0,\quad
\mbox{for\ every}\ m_{1}\in[0,R_{0}]~\mbox{and}\ R_{0}<S^{\frac{N}{4}}.
\end{equation}
Correspondingly, for every $V_{1}\in[0,|V|_{\infty}]$, we denote by $\Upsilon_{V_{1}}$,
$\Upsilon_{V_{1},T}:H^{2}(\mathbb{R}^{N})\rightarrow\mathbb{R}$ the following functionals:
\begin{equation*}
\Upsilon_{V_{1}}=\frac{1}{2}\|\Delta u\|_{2}^{2}+\frac{V_{1}}{2}\int_{\mathbb{R}^{N}}
|u|^{2}dx-\frac{\mu}{q}\int_{\mathbb{R}^{N}}|u|^{q}-\frac{1}{2^{**}}\int_{\mathbb{R}^{N}}|u|^{2^{**}}dx
\end{equation*}
and
\begin{equation*}
\Upsilon_{V_{1},T}(u)=\frac{1}{2}\|\Delta u\|_{2}^{2}+\frac{V_{1}}{2}\int_{\mathbb{R}^{N}}
|u|^{2}dx-\frac{\mu}{q}\int_{\mathbb{R}^{N}}|u|^{q}-\frac{\eta\|\Delta u\|_{2}}{2^{**}}\int_{\mathbb{R}^{N}}|u|^{2^{**}}dx.
\end{equation*}

\section{Energy Functional $\Upsilon_{V_{1},T}$}\label{S3}

In this section, we will establish the  properties of the functional
$\Upsilon_{V_{1},T}$ restricted on $S_{c}$.
\begin{lemma}\label{lem3.1}
The energy functional $\Upsilon_{V_{1},T}$ is bounded from below on
$S_{c}$ and coercive.
\end{lemma}
\begin{proof}
For each $u\in S_{c}$, we have the following  properties
of $H_{T,c}$:
$$\Upsilon_{V_{1},T}(u)\geq H_{T,c}(\|\Delta u\|_{2})\geq\inf\limits_{m>0}H_{T,c}(m)>-\infty.$$
Moreover, $\Upsilon_{V_{1},T}(u)\rightarrow\infty,$  as $\|\Delta u\|_{2}\rightarrow\infty$.
This completes the proof of Lemma \ref{lem3.1}.
\end{proof}
Next, we can define
\begin{equation}\label{e3.1}
\tilde{d}_{V_{1},T,c}=\inf_{u\in S_{c}}\Upsilon_{V_{1},T}(u).
\end{equation}
The following lemma yields an important property of $\tilde{d}_{V_{1},T,c}$.
\begin{lemma}\label{lem3.2}
For every $c\leq\bar{c}$, there exists $V_{*}>0$ such that $\tilde{d}_{V_{1},T,c}<0$
if $V_{1}<V_{*}$.
\end{lemma}
\begin{proof}
Fixing $u\in S_{c}$, we set
$$\mathcal{H}(x,\alpha)(x)=\frac{N\alpha}{2}u(e^{\alpha}x),\quad \mbox{for\ every}\ \alpha\in\mathbb{R}.$$
A direct computation gives
\begin{align*}
&\int_{\mathbb{R}^{N}}|\mathcal{H}(u,\alpha)(x)|^{2}dx=c^{2},\\
&\int_{\mathbb{R}^{N}}|\mathcal{H}(u,\alpha)(x)|^{\varepsilon}dx=e^{\frac{N\alpha(\varepsilon-2)}{2}}
\int_{\mathbb{R}^{N}}|u|^{\varepsilon}dx,\quad \hbox{ for every }  \varepsilon\geq2,
\end{align*}
and
\begin{equation}\label{e3.3}
\int_{\mathbb{R}^{N}}|\Delta\mathcal{H}(u,\alpha)(x)|^{2}dx=e^{4\alpha}
\int_{\mathbb{R}^{N}}|\Delta u|^{2}dx,
\end{equation}
which leads to
\begin{eqnarray*}
\Upsilon_{V_{1},T}(u) &\geq&\nonumber
\frac{1}{2}\int_{\mathbb{R}^{N}}|\mathcal{H}(u,\alpha)(x)|^{2}dx+\frac{V_{1}\bar{c}}{2}-
\frac{\mu}{q}\int_{\mathbb{R}^{N}}|\mathcal{H}(u,\alpha)(x)|^{q}dx\\
&=&\frac{e^{4\alpha}}{2}\int_{\mathbb{R}^{N}}|\Delta u(x)|^{2}dx+\frac{V_{1}\bar{c}}{2}
-\frac{\mu e^{\alpha\gamma_{q}}}{q}\int_{\mathbb{R}^{N}}|u(x)|^{q}dx.
\end{eqnarray*}
Since $q<\frac{8+2N}{N}$, there exists $\beta<0$ such that
$$\frac{e^{4\alpha}}{2}\int_{\mathbb{R}^{N}}|\Delta u(x)|^{2}dx
-\frac{\mu e^{\alpha\gamma_{q}}}{q}\int_{\mathbb{R}^{N}}|u(x)|^{q}dx:=B_{\alpha}<0.$$
Therefore, setting $V_{1}<V_{*}:=-\frac{B_{\alpha}}{\bar{c}}$, it follows that
$$\Upsilon_{V_{1},T}(u)(\mathcal{H}(u,\alpha))<B_{\alpha}-\frac{B_{\alpha}}{2}=\frac{B_{\alpha}}{2}<0,$$
which implies that $\tilde{d}_{V_{1},T,c}(u)<0$.
This completes the proof of Lemma \ref{lem3.2}.
\end{proof}
Hereafter, we will always assume that $V_{1}<V_{*}$ holds. The proof of the
following lemma is standard -- see the proof of Li et al. \cite[Lemma 3.3]{xx}.
\begin{lemma}\label{lem3.3}
The energy functional $\Upsilon_{V_{1},T}(u)$ has the following properties:
 \begin{itemize}
        \item[$(i)$] $\Upsilon_{V_{1},T}\in C^{1}(H^{2}(\mathbb{R}^{N}),\mathbb{R})$.
   \item[$(ii)$]  If $\Upsilon_{V_{1},T}\leq0$, then $\|\Delta u\|_{2}<R_{0}$. Moreover,
 $\Upsilon_{V_{1},T}(v)=\Upsilon_{V_{1}}(v)$, for every  $v$ in a sufficiently small
neighborhood of $u$ in $H^{2}(\mathbb{R}^{N})$.
    \end{itemize}
\end{lemma}
Next, we recall
$$H_{T}(c,m)=\frac{1}{2}m^{2}-\frac{\mu}{q}C_{N,q}^{q}c^{\frac{4q-(q-2)N}{8}}m^{\frac{(q-2)N}{4}}-
\frac{\eta(m)}{2^{**}}S^{-\frac{2^{**}}{2}}m^{2^{**}},\quad t>0, $$
and define
$$F_{T}(c,m):=\frac{1}{2}-\frac{\mu}{q}C_{N,q}^{q}c^{\frac{4q-(q-2)N}{8}}m^{\frac{(q-2)N-8}{4}}-
\frac{\eta(m)}{2^{**}}S^{-\frac{2^{**}}{2}}m^{2^{**}-2}.$$
Consider $F_{T,c}(m)$, which is defined on $(0,\infty)$ by
$m\mapsto F_{T}(a,m)$,
and let
$$F(c,m)=\frac{1}{2}-\frac{\mu}{q}C_{N,q}^{q}c^{\frac{4q-(q-2)N}{8}}m^{\frac{(q-2)N-8}{4}}-
\frac{1}{2^{**}}S^{-\frac{2^{**}}{2}}m^{2^{**}-2}.$$
\begin{lemma}\label{lem3.4}
Suppose that $(c_{2},m_{2})\in(0,\infty)\times(0,\infty)$ satisfies $F(c_{2},t_{2})\geq0$.
Then, for every  $c_{1}\in(0,c_{2}]$, we have that
$$F_{T}(c_{1},m_{1})\geq0 \quad \mbox{if}\ m_{1}\in\left[\left(\frac{c_{1}}{c_{2}}\right)^{\frac{1}{2}}t_{2},t_{2}\right].$$
\end{lemma}
\begin{proof}
Since $c\mapsto F_{T}(c,m)$ is a non-increasing function, it is easy to obtain that
$$F_{T}(c_{1},m_{1})\geq F_{T}(c_{2},m_{2})\geq F(c_{2},m_{2})\geq0.$$
By a direct calculation, we can get that
\begin{align*}
F_{T}(c_{1},(\frac{c_{1}}{c_{2}})^{\frac{1}{2}}m_{2}) 
\geq&
F(c_{1},(\frac{c_{1}}{c_{2}})^{\frac{1}{2}}m_{2})\\
=& \frac{1}{2}-\frac{\mu}{q}C_{N,q}^{q}(\frac{c_{1}}{c_{2}})^{\frac{q-2}{2}}c_{2}^{\frac{4q-(q-2)N}{8}}m_{2}^{\frac{(q-2)N-8}{4}}\\
&-\frac{1}{2^{**}}S^{-\frac{2^{**}}{2}}(\frac{c_{1}}{c_{2}})^{\frac{2^{**}-2}{2}}m_{2}^{2^{**}-2}\\
\geq&
\frac{1}{2}-\frac{\mu}{q}C_{N,q}^{q}c_{2}^{\frac{4q-(q-2)N}{8}}m_{2}^{\frac{(q-2)N-8}{4}}-
\frac{1}{2^{**}}S^{-\frac{2^{**}}{2}}m_{2}^{2^{**}-2}\\
=& F(c_{2},m_{2})\geq0.
\end{align*}
Hence,
\begin{equation*}
F_{T}(c_{1},(\frac{c_{1}}{c_{2}})m_{2})\geq0 \quad \mbox{and}\quad F_{T}(c_{1},m_{2})\geq0,
\end{equation*}
so, by definition of $\eta$, we get $$F_{T}(c_{1},m_{1})\geq0,  \quad \mbox{for every}\  m_{1}\in\left[\left(\frac{c_{1}}{c_{2}}\right)^{\frac{1}{2}}m_{2},m_{2}\right].$$
This completes the proof of Lemma \ref{lem3.4}.
\end{proof}

\begin{lemma}\label{lem3.5}
For every $u\in S_{c}$, the following holds:
$$\Upsilon_{V_{1},T}(u)\geq\|\Delta u\|_{2}^{2}F_{t}(c,\|\Delta u\|_{2}).$$
\end{lemma}
\begin{proof}
Applying  Lemma \ref{lem2.1}, we obtain that
\begin{eqnarray*}
\Upsilon_{V_{1},T}(u) &\geq&\nonumber
\frac{1}{2}\|\Delta u\|_{2}^{2}-\frac{\mu}{q}C_{N,q}^{q}c^{\frac{4q-(q-2)N}{8}}
\|\Delta u\|_{2}^{\frac{(q-2)N}{4}}-\frac{\eta(\|\Delta u\|_{2})}{2^{**}}S^{-\frac{2^{**}}{2}}\|\Delta u\|_{2^{**}}^{2^{**}}\\
&=&\|\Delta u\|_{2}^{2}\left[\frac{1}{2}-\frac{\mu}{q}C_{N,q}^{q}c^{\frac{4q-(q-2)N}{8}}
\|\Delta u\|_{2}^{\frac{(q-2)N-8}{4}}-\frac{\eta(\|\Delta u\|_{2})}{2^{**}}S^{-\frac{2^{**}}{2}}\|\Delta u\|_{2^{**}}^{2^{**}-2}\right]\\
&=&\|\Delta u\|_{2}^{2}F_{T}(c,\|\Delta u\|_{2}).
\end{eqnarray*}
This completes the proof of Lemma \ref{lem3.5}.
\end{proof}
The proof of the next lemma is standard---see e.g.,  the proof of Hu and Mao \cite[Lemma 2.3]{hm}, so we will omit it. Recall that
 $\tilde{d}_{V_{1},T,c}$ was defined in \eqref{e3.1}.
\begin{lemma}\label{lem3.6}
$\tilde{d}_{V_{1},T,c}$ is continuous with regard to $c\in(0,\bar{c}]$.
\end{lemma}
For every $c<\bar{c}$, in view of $H(\bar{c},m)=mF(\bar{c},m)$, it follows that
$F(\bar{c},R_{0})=0$. Moreover, $c\rightarrow F(c,m)$ is a non-increasing function,
so
it follows that $F(c,R_{0})\geq0$, for every  $c\in(0,\bar{c}]$.

\begin{lemma}\label{lem3.7}
We have $\frac{c_{1}}{c_{2}}\tilde{d}_{V_{1},T,c_{2}}<\tilde{d}_{V_{1},T,c_{1}}<0$,
where $0<c_{1}<c_{2}\leq\bar{c}$.
\end{lemma}
\begin{proof}
Set $\varsigma=(\frac{c_{2}}{c_{1}})^{\frac{1}{2}}.$ Then $\varsigma>1$. Let
$\{u_{n}\}\subset S_{c_{1}}$ be a minimizing sequence with respect to $\tilde{d}_{V_{1},T,c_{1}}$,
that is, $\Upsilon_{V_{1},T}(u_{n})\rightarrow\tilde{d}_{V_{1},T,c_{1}}<0$, as $n\rightarrow+\infty$
by Lemma \ref{lem3.2}. Therefore, there exists $n_{0}$ such that
\begin{equation}\label{e3.5}
\Upsilon_{V_{1},T}(u_{n})<0, \quad  \mbox{for every}\ ~ n\geq n_{0}.
\end{equation}
In view of Lemma \ref{lem3.4} and $F_{T}(c_{2},R_{0})\geq0$, we have that  $F_{T}(c_{1},m)\geq0$,
for every  $m\in[(\frac{c_{2}}{c_{1}})^{\frac{1}{2}}R_{0},R_{0}]$.
Hence, we can deduce from \eqref{e3.5} and Lemma \ref{lem3.5} that
\begin{equation*}
\|\Delta u\|_{2}<\left(\frac{c_{2}}{c_{1}}\right)^{\frac{1}{2}}R_{0},  \quad  \mbox{for every}\ ~ n\geq n_{0}.
\end{equation*}
Setting $v_{n}=\varsigma u_{n}$, we get $v_{n}\in S_{c_{2}}$. By \eqref{e3.3},
one has $\|\Delta v_{n}\|_{2}=\varsigma\|\Delta u_{n}\|_{2}<R_{0}$.
Therefore,  $$\eta(\|\Delta u_{n}\|_{2})=\eta(\|\Delta v_{n}\|_{2})=1.$$
Through direct calculations, we find that
\begin{eqnarray*}
\tilde{d}_{V_{1}T,c_{2}} &\geq&\nonumber
\Upsilon_{V_{1},T}(v_{n})\\
&=& \varsigma^{2}\Upsilon_{V_{1},T}(u_{n})+\frac{\varsigma^{2}-\varsigma^{q}}{q}\mu|u_{n}|_{q}^{q}
+\frac{\eta(\|\Delta u_{n}\|_{2})\varsigma^{2}-\eta(\|\Delta v_{n}\|_{2})\varsigma\varsigma^{2^{**}}}{2^{**}}|u_{n}|_{2^{**}}^{2^{**}} \\
&=& \varsigma^{2}\Upsilon_{V_{1},T}(u_{n})+\frac{\varsigma^{2}-\varsigma^{q}}{q}\mu|u_{n}|_{q}^{q}
+\frac{\varsigma^{2}-\varsigma^{2^{**}}}{2^{**}}|u_{n}|_{2^{**}}^{2^{**}}.
\end{eqnarray*}
For each $m\in(2,\frac{2N}{N-4})$, there exist positive constants $C$ and $n_{0}$
such that $|u_{n}|_{m}^{m}\geq C$, for every $n\geq n_{0}$.
If not, there exists $m_{1}\in(2,\frac{2N}{N-4})$ such that $|u_{n}|_{m_{1}}^{m_{1}}\rightarrow0$
as $n\rightarrow+\infty$; then, by the vanishing lemma in Lions \cite{P.Lion1, P.Lion2}, $|u_{n}|_{q}^{q}\rightarrow0$
as $n\rightarrow+\infty$. By \eqref{e2.3}, it follows that
$$0>\tilde{d}(V_{1},T,c_{2})=\Upsilon_{V_{1},T}(u_{n})\geq-\frac{\mu}{q}|u_{n}|^{q}_{q}\rightarrow 0,\quad 
 \mbox{as}\ ~ n\rightarrow\infty,$$
which is a contradiction, and so, the assertion holds. We can deduce that, for sufficiently large $n\in\mathbb{N}$,
$$\tilde{d}_{V_{1},T,c_{2}}\leq\varsigma^{2}\Upsilon_{V_{1},T}(u_{n})+\frac{\varsigma^{2}-\varsigma^{q}C}{q}.$$
Apply Lemma \ref{lem3.6} and let $n\rightarrow+\infty$ to conclude that
$\tilde{d}_{V_{1},T,c_{2}}<\varsigma^{2}\tilde{d}_{V_{1},T,c_{1}}$,  which implies that
$$\frac{c_{1}}{c_{1}}{\tilde{d}_{V_{1},T,c_{2}}}<\tilde{d}_{V_{1},T,c_{1}}.$$
This completes the proof of Lemma \ref{lem3.7}.
\end{proof}

\begin{lemma}\label{lem3.8}
If $\hat{d}_{V_{1},T,c}<0$, then $u_{n}\rightarrow u$ in $L^{2^{**}}(\mathbb{R}^{N})$.
\end{lemma}
\begin{proof}
In fact, we have $\|\Delta u\|_{2}\geq R_{0}$, for sufficiently large $n$, by Lemma \ref{lem2.2}
and \ref{lem2.3}, and there exist two positive measures $\nu,\kappa\in\mathcal{M}(\mathbb{R}^{N})$ such that
\begin{equation}\label{e3.7}
\|\Delta u\|_{2}^{2}\rightharpoonup \kappa ~ \mbox{and}\ ~ \|u_{n}\|^{2^{**}}\rightharpoonup \nu
\hbox{ in }
\mathcal{M}(\mathbb{R}^{N}),\quad 
\hbox{  as }
n\rightarrow\infty.
\end{equation}
We define that $\hat{\varpi}_{\rho}(x):=\hat{\varpi}(\frac{x-x_{i}}{\rho})$ as a cut-off function,
where $\hat{\varpi}\in C_{c}^{\infty}(\mathbb{R}^{N})$, $\hat{\varpi}=1$ in $B_{1}$,
$\hat{\varpi}=0$ in $B_{2}^{c}$, and $\|\Delta \hat{\varpi}\|_{L^{\infty}(\mathbb{R}^{N})}\leq2$.
Noting that $\{u_{n}\hat{\varpi}_{\rho}\}$ is bounded in $H^{2}(\mathbb{R}^{N})$ and
$\hat{\varpi}_{\rho}$ takes values in $\mathbb{R}$, we have
$\langle \Upsilon'(u_{n}),\hat{\varpi}_{\rho}u_{n}\rangle\rightarrow0$, as $n\rightarrow\infty$,
for sufficiently large $n\in\mathbb{N}$.  In view of this and $(V_{1})$,
we can deduce that
\begin{align*}
&\int_{\mathbb{R}^{N}}|\Delta u|^{2}\hat\varpi_{\rho}(x)dx+\int_{\mathbb{R}^{N}}
u_{n}|\Delta u_{n}|\Delta\hat{\varpi}_{\rho}(x)dx+2\int_{\mathbb{R}^{N}}|\Delta u_{n}|\nabla
u_{n}\nabla\hat{\varpi}_{\rho}(x)dx  \\
&=\mu\int_{\mathbb{R}^{N}}|u_{n}|^{q}\hat{\varpi}_{\rho}(x)dx+\int_{\mathbb{R}^{N}}
|u_{n}|^{2^{**}}\hat{\varpi}_{\rho}(x)dx+o_{n}(1).
\end{align*}
Using the H\"{o}lder inequality, we get
\begin{eqnarray*}
\limsup_{n\rightarrow\infty}\left|\int_{\mathbb{R}^{N}}\Delta u _{n}\Delta\hat{\varpi}_{\rho}u_{n}dx\right| &\leq&\nonumber
\limsup_{n\rightarrow\infty}\left(\int_{\mathbb{R}^{N}}|\Delta u _{n}|^{2}dx\right)^{\frac{1}{2}}
\left(\int_{\mathbb{R}^{N}}|u _{n}\Delta\hat{\varpi}_{\rho}|^{2}dx\right)^{\frac{1}{2}}\\
&\leq& C_{1}\left(\int_{B_{2\rho}}|u_{n}|^{2}|\Delta\hat{\varpi}_{\rho}|^{2}\right)^{\frac{1}{2}}\\
&\leq& C_{1}\left(\int_{B_{2\rho}}|\Delta\hat{\varpi}_{\rho}|^{\frac{N}{2}}dx\right)^{\frac{2}{N}}
\left(\int_{B_{2\rho}}|u|^{2^{**}}dx\right)^{\frac{1}{2^{**}}}  \\
&\leq& C_{2}\left(\int_{B_{2\rho}}|u|^{2^{**}}dx\right)^{\frac{1}{2^{**}}}\rightarrow0, \quad ~\mbox{as}\ ~\rho\rightarrow0.
\end{eqnarray*}
Similarly, by the definition of $\hat{\varpi}_{\rho}$
and the boundedness of $\{u_{n}\}$, we have
$$\lim\limits_{\rho\rightarrow0}\limsup\limits_{n\rightarrow\infty}\int_{\mathbb{R}^{N}}\Delta u_{n}\nabla u_{n}\nabla
\hat{\varpi}_{\rho}dx=0.$$
Recalling that $q\in(2,2+\frac{8}{N})$ and the definition of $\hat{\varpi}_{\rho}$, we have
$$\lim\limits_{\rho\rightarrow0}\limsup\limits_{n\rightarrow\infty}\int_{\mathbb{R}^{N}}|u|^{q}
\hat{\varpi}_{\rho}dx=0,$$
so  by \eqref{e3.7} and Lemma \ref{lem2.2},
$$\lim\limits_{\rho\rightarrow0}\lim\limits_{n\rightarrow\infty}\int_{\mathbb{R}^{N}}|\Delta u_{n}|^{2}
\hat{\varpi}_{\rho}dx=\lim_{\rho\rightarrow0}\int_{\mathbb{R}^{N}}\hat{\varpi}_{\rho}dk=\kappa(\{x_{i}\})=\kappa_{i},$$
$$\lim\limits_{\rho\rightarrow0}\lim\limits_{n\rightarrow\infty}\int_{\mathbb{R}^{N}}|u|^{2^{**}}
\hat{\varpi}_{\rho}dx=\lim_{\rho\rightarrow0}\int_{\mathbb{R}^{N}}\hat{\varpi}_{\rho}dv=\nu(\{x_{i}\})=\nu_{i}.$$
Letting $\rho\rightarrow0$, we can deduce that $\kappa_{i}=\nu_{i}.$ Moreover, we obtain $\kappa_{i}\geq S\kappa_{i}\frac{2}{2^{**}},$
which implies that
\begin{equation*}
\mbox{(i)}\ ~ \kappa_{i}=0 \quad  \mbox{or}\quad \mbox{(ii)}\ ~ \kappa_{i}\geq S^{\frac{4}{N}}.
\end{equation*}
Assume, to the contrary,  that there existed $i_{0}\in I$ such that
$\kappa_{i_{0}}\geq S^\frac{N}{4}.$
Then, we would obtain that
$$R_{0}^{2}\geq\lim\limits_{\rho\rightarrow0}\lim\limits_{n\rightarrow\infty}\|\Delta u_{n}\|_{2}^{2}\geq\lim\limits_{\rho\rightarrow0}\lim\limits_{n\rightarrow\infty}\int_{\mathbb{R}^{N}}
|\Delta u_{n}|^{2}\nabla\hat{\varpi}_{\rho}dx=\lim\limits_{\rho\rightarrow0}
\int_{\mathbb{R}^{N}}\hat{\varpi}_{\rho}d\kappa\geq S^{\frac{N}{4}} $$
which would contradict $\eqref{e2.3}$. So, indeed,
$$u_{n}\rightarrow u \quad \mbox{in}\ ~ L_{loc}^{2^{**}}(\mathbb{R}^{N}).$$
In addition, we define a cut-off function, $\phi\in C^{\infty}(\mathbb{R}^{N})$
to be such that
 $\phi=0$ in $B_{1}$ and $\phi=1$ in $\mathbb{R}^{N}\setminus B_{2}$;
we set $\phi_{R}(x)=\phi(\frac{x}{R})$. Note that $\{u_{n}\phi_{R}\}$ is bounded on
$H^{2}(\mathbb{R}^{N})$ and $\phi_{R}$ takes values in $\mathbb{R}$, so a direct calculation
shows that $\langle \Upsilon'(u_{n}\phi_{R}),u_{n}\phi_{R}\rangle\rightarrow0$,  as
$n\rightarrow\infty$, for sufficiently large $n\in\mathbb{N}$.  According
to these facts, we can deduce that
\begin{align*}
&\nonumber\int_{\mathbb{R}^{N}}|\Delta u|^{2}\phi_{R}dx+\int_{\mathbb{R}^{N}}
u_{n}|\Delta u_{n}|\Delta\phi_{R}(x)dx+2\int_{\mathbb{R}^{N}}|\Delta u_{n}|\nabla
u_{n}\nabla\phi_{R}(x)dx  \\
&=\mu\int_{\mathbb{R}^{N}}|u_{n}|^{q}\phi_{R}(x)dx+\int_{\mathbb{R}^{N}}
|u_{n}|^{2^{**}}\phi_{R}(x)dx+o_{n}(1).
\end{align*}
It is easy to see that that
$$\lim\limits_{R\rightarrow\infty}\lim\limits_{n\rightarrow\infty}\int_{\mathbb{R}^{N}}u_{n}|\Delta u_{n}|
\Delta\phi_{R}dx=0$$
and
$$\lim\limits_{R\rightarrow\infty}\lim\limits_{n\rightarrow\infty}\int_{\mathbb{R}^{N}}|\Delta u_{n}|\nabla
u_{n}\nabla\phi_{R}(x)dx=0.$$
By the definition of $\phi_{R}$, one has
$$\int_{\{x\in\mathbb{R}^{N}:|x|>R\}}|\Delta u_{n}|^{2}dx\leq\int_{\mathbb{R}^{N}}
\phi_{R}|\Delta u_{n}|^{2}dx\leq\int_{\{x\in\mathbb{R}^{N}:|x|>\frac{R}{2}\}}|\Delta u_{n}|^{2}dx.$$
Thus, by Lemma \ref{lem2.3}, we have
\begin{equation*}
\lim\limits_{R\rightarrow\infty}\lim\limits_{n\rightarrow\infty}\int_{\mathbb{R}^{N}}\phi_{R}|\Delta u_{n}|^{2}
dx=\kappa_{\infty}.
\end{equation*}
Similarly, we obtain that
$$\lim\limits_{R\rightarrow\infty}\lim\limits_{n\rightarrow\infty}\int_{\mathbb{R}^{N}}\phi_{R}
|u_{n}|^{2^{**}}dx=\nu_{\infty}$$
and
$$\lim\limits_{R\rightarrow\infty}\lim\limits_{n\rightarrow\infty}\int_{\mathbb{R}^{N}}
\phi_{R}|u_{n}|^{q}dx=\lim\limits_{R\rightarrow\infty}\int_{\mathbb{R}^{N}}\phi_{R}|u|^{q}dx
\lim\limits_{R\rightarrow\infty}\int_{|x>\frac{R}{2}|}\phi_{R}|u|^{q}dx=0.$$
Letting $R\rightarrow\infty$, we deduce that $\kappa_{\infty}=\nu_{\infty}$.
Furthermore, we obtain $\kappa_{\infty}\geq S\kappa_{\infty}^{\frac{N}{4}}$ and get that
\begin{equation}\label{e3.12}
\mbox{(iii)}\ ~ \kappa_{\infty}=0 \quad \mbox{or}\quad \mbox{(iv)}\ ~\kappa_{\infty}\geq S^{\frac{N}{4}}.
\end{equation}
Similarly, we claim that $(iv)$ cannot occur. So, we have
$$u_{n}\rightarrow u \quad \mbox{in}\ ~ L^{2^{**}}(\mathbb{R}^{N}\setminus B_{R}(0));$$
hence, we know that
$$u_{n}\rightarrow u \quad \mbox{in}\ ~ L^{2^{**}}(\mathbb{R}^{N}).$$
This completes the proof of Lemma \ref{lem3.8}.
\end{proof}

\begin{lemma}\label{lem3.9}
Let $\{u_{n}\}\subset S_{c}$ be a minimizing sequence with respect to
$\tilde{d}_{V_{1},T,c}$. Then, for some subsequence,

$(i)$ either $\{u_{n}\}$ is strongly convergent,

$(ii)$ or there exists $(y_{n})\subset\mathbb{R}^{N}$ with $(y_{n})\rightarrow\infty$
such that the sequence $\bar{u}_{n}=u_{n}(x+y_{n})$ is strongly convergent to a function
$\bar{u}\in S_{c}$ with $\Upsilon_{V_{1},T,c}(\bar{u})=\tilde{d}_{V_{1},T,c}$.
\end{lemma}
\begin{proof}
Observing that $\{u_{n}\}$ is bounded in $H^{2}(\mathbb{R}^{N})$ by Lemma \ref{lem3.1}
and Lemma \ref{lem3.2}, 
we can conclude that
there exists $u\in H^{2}(\mathbb{R}^{N})$ such that $u_{n}\rightharpoonup u$
in $H^{2}(\mathbb{R^{N}})$ for some subsequence. Now, let us study all three possibilities as follows.

\textbf{Case 1:} $u\not\equiv0$ and $\|u\|_{2}^{2}=b<c$.

In this case, letting $v_{n}:=u_{n}-u$, $e_{n}=\|v_{n}\|^{2}_{2}$, one has that
$\|v_{n}\|_{2}^{2}\rightarrow e$, where $c=b+e$. It follows by the Fatou Lemma and
the Br\'{e}zis-Lieb Lemma (see Willem \cite{wm}) that $$\|\Delta u\|_{2}^{2}\leq\liminf\limits_{n\rightarrow\infty}
\|\Delta u_{n}\|_{2}^{2},$$ so we have
$$\|u_{n}\|_{2}^{2}=\|v_{n}\|_{2}^{2}+\|u\|_{2}^{2}+o_{n}(1)$$
and
$$\|\Delta u _{n}\|_{2}^{2}=\|\Delta u_{n}\|_{2}^{2}+\|v_{n}\|_{2}^{2}+o_{n}(1).$$
 Noting that $e_{n}\in(0,c)$ for sufficiently large $n$, and the fact that $\eta$ is continuous, non-increasing and Lemma
\ref{lem3.7}, one has
\begin{eqnarray*}
\tilde{d}_{V_{1},T,c}+o_{n}(1) &=&\nonumber
\Upsilon_{V_{1},T,c}(u_{n})\\
&=& \frac{1}{2}\|\Delta v_{n}\|_{2}^{2}+\frac{1}{2}V_{1}\|v_{n}\|_{2}^{2}-\frac{\mu}{q}
\|v_{n}\|_{q}^{q}-\frac{\eta(\|\Delta v_{n}\|_{2}^{2})}{2^{**}}\|v_{n}\|^{2^{**}}_{2^{**}}\\
&\quad&+\frac{1}{2}\|\Delta u\|_{2}^{2}+\frac{1}{2}V_{1}\|u\|_{2}^{2}-\frac{\mu}{q}
\|u\|_{q}^{q}-\frac{\eta(\|\Delta u\|_{2}^{2})}{2^{**}}\|u\|^{2^{**}}_{2^{**}}+o_{n}(1)\\
&\geq& \Upsilon_{V_{1},T,c}(v_{n})+\Upsilon_{V_{1},T,c}(u)+o_{n}(1)\\
&\geq&\tilde{d}_{V_{1},T,e_{n}}+\tilde{d}_{V_{1},T,b}+o_{n}(1)\\
&\geq&\frac{e_{n}}{c}\tilde{d}_{V_{1},T,c}+\tilde{d}_{V_{1},T,b}+o_{n}(1).
\end{eqnarray*}
By Lemma \ref{lem3.7}, letting $n\rightarrow+\infty$, we obtain that
$$\tilde{d}_{V_{1},T,c}\geq\frac{e}{c}\tilde{d}_{V_{1},T,c}+\tilde{d}_{V_{1},T,b}
>\frac{e}{c}\tilde{d}_{V_{1},T,c}+\frac{b}{c}\tilde{d}_{V_{1},T,c}=\tilde{d}_{V_{1},T,c},$$
which is a contradiction.

\textbf{Case 2:} $\|u_{n}\|_{2}^{2}=c$.

In this case, we have $u_{n}\rightarrow u$ in $L^{2}(\mathbb{R}^{N})$
and $u_{n}\rightarrow u$ in $L^{m}(\mathbb{R}^{N})$, for every $m\in(2,\frac{2N}{N-4})$.
By Lemma \ref{lem3.8} and the Fatou lemma, we obtain that
\begin{eqnarray*}
\tilde{d}_{V_{1},T,c} &=&\nonumber
\lim_{n\rightarrow+\infty}\Upsilon_{V_{1},T,c}(u_{n})\\
&=& \lim_{n\rightarrow+\infty}\left(\frac{1}{2}\|\Delta u_{n}\|_{2}^{2}+\frac{1}{2}V_{1}\|u_{n}\|_{2}^{2}-\frac{\mu}{q}
\|u_{n}\|_{q}^{q}-\frac{\eta(\|\Delta u_{n}\|_{2}^{2})}{2^{**}}\|u_{n}\|^{2^{**}}_{2^{**}}\right)\\
&\geq& \Upsilon_{V_{1},T}(u).
\end{eqnarray*}
Since $u\in S_{c}$, we can conclude that $\Upsilon_{V_{1},T}(u)=\tilde{d}_{V_{1},T,c}$, so
$u_{n}\rightarrow u$ in $H^{2}(\mathbb{R}^{N})$, which implies that $(i)$ indeed occurs.

\textbf{Case 3:} $u\equiv0.$

In this case, $u_{n}\rightharpoonup0$ in $H^{2}(\mathbb{R}^{N})$, we
assert that there exist $R',k_{1}>0$ and the sequence $\{y_{n}\}\subset\mathbb{R}^{N}$ such that, 
for every $n$,
\begin{equation*}
\int_{B_{R'}(y_{n})}|u_{n}|^{2}\geq k_{1}.
\end{equation*}
Otherwise, this would imply that $u_{n}\rightarrow0$ in $L^{m}(\mathbb{R}^{N})$,
for every $m\in(2,\frac{2N}{N-2})$ by the  vanishing lemma. As a result, we would have $\Upsilon_{V_{1},T}(u_{n})
\geq\frac{1}{2}\|\Delta u_{n}\|_{2}^{2}+o_{n}(1)$. However, this would contradict the fact that
$\Upsilon_{V_{1},T}(u_{n})\rightarrow\tilde{d}(V_{1},T,c)<0$.

Hence, in all cases, \eqref{e3.12}
holds, and obviously, $|y_{n}|\rightarrow+\infty.$
 Consequently, by defining $\bar{u}_{n}(x)=u_{n}
(x+y_{n})$, it is evident that $\{\bar{u}_{n}\}\subset S_{c}$, so we also know that it is a minimizing sequence with
respect to $\tilde{d}_{V_{1},T,c}$. Furthermore, there exists $\bar{u}\in H^{2}(\mathbb{R}^{N})$
such that $\bar{u}_{n}\rightharpoonup \bar{u}$ in $H^{2}(\mathbb{R}^{N})$. Following the approach used in the
first two cases, we can conclude that $\bar{u}_{n}\rightarrow \bar{u}$ in $H^{2}(\mathbb{R}^{N})$, thus
confirming the validity of $(ii).$
This completes the proof of Lemma \ref{lem3.10}.
\end{proof}
\begin{lemma}\label{lem3.10}
$\tilde{d}_{V_{1},T,c}$ is attained.
\end{lemma}
\begin{proof}
By Lemma \ref{lem3.1} and Lemma \ref{lem3.9}, there exists a bounded minimizing sequence
$\{u_{n}\}\subset S_{c}$ and $u_{n}\rightarrow u$ in $H^{2}(\mathbb{R}^{N})$ with respect
to $\tilde{d}_{V_{1},T,c}=\Upsilon_{V_{1},T}<0$. Then, $\{u_{n}\}$ is also a minimizing sequence
for $\Upsilon_{V_{1}}(u)$ and $\tilde{d}_{V_{1},T,c}=\Upsilon_{V_{1},T}$ by Lemma \ref{lem3.3}.
This completes the proof of Lemma \ref{lem3.10}.
\end{proof}
A direct result of Lemma \ref{lem3.10} is the following corollary.
\begin{corollary}\label{cor3.1}
If $V_{1}<V_{2}<V_{*},$
then
 $\tilde{d}_{V_{1},T,c}<\tilde{d}_{V_{2},T,c}$.
\end{corollary}
\begin{proof}
Let $u\in S_{c}$ satisfy $\tilde{d}_{V_{2},T,c}=\Upsilon_{V_{2},T}$. Then,
$\tilde{d}_{V_{1},T,c}\leq\Upsilon_{V_{1},T}<\Upsilon_{V_{1},T}=\tilde{d}_{V_{2},T,c}$.
\end{proof}

\section{Energy Functional $\Upsilon_{\varepsilon,T}(u)$}\label{S4}
In this section we will establish key properties of the following functional
$\Upsilon_{\varepsilon,T}(u):H^{2}(\mathbb{R}^{N})\rightarrow\mathbb{R},$  given by
$$\Upsilon_{\varepsilon,T}(u)=\frac{1}{2}\|\Delta u\|_{2}^{2}+\frac{1}{2}\int_{\mathbb{R}^{N}}V(\varepsilon x)|u|^{2}
-\frac{\mu}{q}\|u\|_{q}^{q}-\frac{\eta(\|\Delta u\|_{2}^{2})}{2^{**}}\|u\|^{2^{**}}_{2^{**}}.$$
More precisely, we will study the following minimum value:
$$\tilde{d}_{\varepsilon,T,c}:=\inf\limits_{u\in S_{c}}\Upsilon_{\varepsilon,T}(u),$$
where $\tilde{d}_{\varepsilon,T,c}$ is well defined by the properties of $H_{T,c}(m)$.
We will denote by $\Upsilon_{0,T}$, $\Upsilon_{\infty,T}:H^{2}(\mathbb{R}^{N})
\rightarrow\mathbb{R}$ the following functionals:
$$\Upsilon_{0,T}(u)=\frac{1}{2}\|\Delta u\|_{2}^{2}
-\frac{\mu}{q}\|u\|_{q}^{q}-\frac{\eta(\|\Delta u\|_{2}^{2})}{2^{**}}\|u\|^{2^{**}}_{2^{**}}$$
and
$$\Upsilon_{\infty,T}(u)=\frac{1}{2}\|\Delta u\|_{2}^{2}+\frac{1}{2}\int_{\mathbb{R}^{N}}V_{\infty}|u|^{2}dx
-\frac{\mu}{q}\|u\|_{q}^{q}-\frac{\eta(\|\Delta u\|_{2}^{2})}{2^{**}}\|u\|^{2^{**}}_{2^{**}}.$$
By $(V_{1})-(V_{3})$, $V_{\infty}<V_{*},$ and Lemma \ref{lem3.10}, the minimum values
$\tilde{d}_{0,T,c}$ and $\tilde{d}_{\infty,T,c}$ defined by
$$\tilde{d}_{0,T,c}:=\inf\limits_{u\in S_{c}}\Upsilon_{0,T}(u) \quad \mbox{and}\quad
\tilde{d}_{\infty,T,c}:=\inf\limits_{u\in S_{c}}\Upsilon_{\infty,T}(u),$$
respectively, are attained. There exist $u_{0},u_{\infty}\in S_{c}$ such that
$$\Upsilon_{0,T}(u_{0})=\tilde{d}_{0,T,c} \quad \mbox{and} \quad\Upsilon_{\infty,T}=\tilde{d}_{\infty,T,c}.$$
Furthermore, by Corollary \ref{cor3.1} and $V_{0}<V_{\infty}$, we have
$\tilde{d}_{0,T,c}<\tilde{d}_{\infty,T,c}<0$.
\begin{lemma}\label{lem4.1}
We have $\limsup\limits_{\varepsilon\rightarrow0^{+}}\tilde{d}_{\varepsilon,T,c}\leq\tilde{d}_{0,T,c}$.
\end{lemma}
\begin{proof}
By Lemma \ref{lem3.10}, we have $u_{0}\in S_{c}$ with $\Upsilon_{0,T}=\tilde{d}_{0,T,c}$. Then,
$$\tilde{d}_{\varepsilon,T,c}\leq\Upsilon_{\varepsilon,T}(u_{0})=\frac{1}{2}\|\Delta u_{0}\|_{2}^{2}+\frac{1}{2}\int_{\mathbb{R}^{N}}V(\varepsilon x)|u_{0}|^{2}-\frac{\mu}{q}\|u_{0}\|_{q}^{q}-\frac{\eta(\|\Delta u_{0}\|_{2}^{2})}{2^{**}}\|u_{0}\|^{2^{**}}_{2^{**}}.$$
Letting $\varepsilon\rightarrow0^{+}$, we obtain, by the Lebesgue dominated convergence theorem,
$$\limsup\limits_{\varepsilon\rightarrow0^{+}}\tilde{d}_{\varepsilon,T,c}\leq\limsup\limits_{\varepsilon\rightarrow0^{+}}
\Upsilon_{\varepsilon,T}(u_{0})=\Upsilon_{0,T}(u_{0})=\tilde{d}_{0,T,c}.$$
This completes the proof of Lemma \ref{lem4.1}.
\end{proof}
By Lemma \ref{lem4.1} and $\tilde{d}_{0,T,c}<\tilde{d}_{\infty,T,c}$, there exists sufficiently
small $\varepsilon_{0}>0$ such that
$$\tilde{d}_{\varepsilon,T,c}<\tilde{d}_{\infty,T,c}, \quad \mbox{for\ every}\ ~ \varepsilon\in(0,\varepsilon_{0}).$$
Similar to the proof of Lemma \ref{lem4.1}, we can prove the following result.
\begin{lemma}\label{lem4.2}
The energy functional $\Upsilon_{\varepsilon,T}$ has the following properties.
\begin{itemize}
\item[$(i)$]
$\Upsilon_{\varepsilon,T}\in C^{1}(H^{2}(\mathbb{R}^{N}),\mathbb{R})$.
\item[$(ii)$]
If $\Upsilon_{\varepsilon,T}(u)\leq0,$ then $\|\Delta u\|_{2}<R_{0}$. Moreover,  $\Upsilon_{\varepsilon,T}(v)
=\Upsilon_{\varepsilon}(v),$ for every  $v$ in a sufficiently small neighborhood of $u$ in $H^{2}(\mathbb{R}^{N})$.
\end{itemize}
\end{lemma}

\begin{lemma}\label{lem4.3}
The weak limit $u_{\varepsilon}$ of $\{u_{n}\}$ is nontrivial.
\end{lemma}
\begin{proof}
Let $\{u_{n}\}\subset S_{c}$ be a minimizing sequence of $\Upsilon_{\varepsilon,T}(u_{n})$
with respect to any $t<\tilde{d}_{\infty,T,c}<0$. Similar to the proof of Lemma \ref{lem3.1}
and Lemma \ref{lem3.2},
one can prove that  $\{\|\Delta u_{n}\|_{2}\}$ is bounded. Hence, there exist
$u\in H^{2}(\mathbb{R}^{N})$ and a subsequence of $\{u_{n}\}$, still
denoted the same, such that
$$u_{n}\rightharpoonup u_{\varepsilon} ~ \mbox{in}\ ~ H^{2}(\mathbb{R}^{N}) \quad \mbox{and}\quad
u_{n}(x)\rightarrow u_{\varepsilon}(x) ~ \mbox{a.e.\ in}\ ~ \mathbb{R}^{N}.$$
Assume to the contrary that $u_{\varepsilon}=0$. Then,
$$t+o_{n}(1)=\Upsilon_{\varepsilon,T}(u_{n})=\Upsilon_{\infty,T}(u_{n})
+\frac{1}{2}\int_{\mathbb{R}^{N}}(V(\varepsilon x)-V_{\infty})|u_{n}|^{2}dx.$$
By $(V_{2})$, for every  given $\delta>0$, there exists
$R>0$ such that
$$V(x)\geq V_{\infty}-\delta, \quad \mbox{for\ every}\ ~ |x|\geq R;$$
hence, we have
\begin{align*}
t+o_{n}(1)=&\Upsilon_{\varepsilon,T}\\
\geq&\Upsilon_{\varepsilon,T}(u_{n})+\frac{1}{2}
\int_{B_{R\setminus\varepsilon}(0)}(V(\varepsilon x)-V_{\infty})|u_{n}|^{2}dx-\frac{\delta}{2}
\int_{B_{R\setminus\varepsilon}^{c}(0)}|u_{n}|^{2}dx.
\end{align*}
Recalling that $\{u_{n}\}$ is bounded in $H^{2}(\mathbb{R}^{N})$ and $u_{n}\rightarrow0$
in $L^{2}(B_{R\setminus\varepsilon}(0))$, it follows that
\begin{equation}\label{e4.1}
t+o_{n}(1)\geq\Upsilon_{\infty,T}(u_{n})-\delta C\geq\tilde{d}_{\infty,T,c}-\delta C.
\end{equation}
We note that $\delta>0$ is arbitrary, so $t\geq\tilde{d}_{\infty,T,c}$, which
is a contradiction. Therefore, the weak limit $u$ of $\{u_{n}\}$ is indeed nontrivial.
This completes the proof of  Lemma \ref{lem4.3}.
\end{proof}

\begin{lemma}\label{lem4.4}
Let $\{u_{n}\}$ be a $(PS)_{t}$ sequence of $\Upsilon_{\varepsilon,T}$ restricted
to $S_{c}$ with $t<\tilde{d}_{\infty,T,c}$, and let $u_{\varepsilon}$ be the weak limit of
$\{u_{n}\}$ in $H^{2}(\mathbb{R}^{N})$. If $u_{n}\nrightarrow u_{\varepsilon}$ in
$H^{2}(\mathbb{R}^{N})$, then there exists $\beta>0,$ independent of $\varepsilon,$ such that
$$\liminf\limits_{n\rightarrow+\infty}\|u_{n}-u_{\varepsilon}\|_{2}\geq\beta.$$
\end{lemma}
\begin{proof}
By Lemma \ref{lem4.2} and $t<\tilde{d}_{\infty,T,c}<0$, we know that
$\|\Delta u_{n}\|_{2}<R_{0}$, for every sufficiently large $n$.  Consequently,
$\{u_{n}\}$ is also a $(PS)_{t}$ sequence $\Upsilon_{\varepsilon},$ constrained
to $S_{c}$, that is,
$$\Upsilon_{\varepsilon}(u_{n})\rightarrow t ~\mbox{and} ~ \|\Upsilon_{\varepsilon}|'_{S_{c}}(u_{n})\|
_{(H^{2}_{rad}(\mathbb{R}^{N}))^{*}}\rightarrow0 \quad \mbox{as}\ ~ n\rightarrow\infty,$$
where $(H^{2}_{rad}(\mathbb{R}^{N})^{*})$ is the dual space of $H^{2}_{rad}(\mathbb{R}^{N})$.
Introducing the functional 
$$\Psi:H^{2}(\mathbb{R}^{N})\rightarrow\mathbb{R},$$ defined by
$\Psi(u)=\frac{1}{2}\int_{\mathbb{R}^{N}}|u_{n}|^{2}dx$, we have
$$S_{c}=\Psi^{-1}\left(\frac{c}{2}\right). $$
By Proposition 5.12 of Willem \cite{wm},
we deduce that there exists a sequence $\{\lambda_{n}\}\subset\mathbb{R}$ satisfying
\begin{equation}\label{e4.2}
\|\Upsilon'_{\varepsilon}(u_{n})-\lambda_{n}\Psi'(u_{n})\|_{(H^{2}(\mathbb{R}^{N}))^{*}}\rightarrow0,\quad
\mbox{as}\ ~n\rightarrow+\infty.
\end{equation}
Due to the boundedness of the sequence $\{u_{n}\}$ in $H^{2}(\mathbb{R}^{N})$, we see that $u_{n}\rightharpoonup u_{\varepsilon}$,
and we let $v_{n}:=u_{n}-u_{\varepsilon}$. It follows that $\{\lambda_{n}\}$ is also bounded,
so for some subsequence, there exists $\lambda_{\varepsilon}$ such that
$\lambda_{n}\rightarrow\lambda_{\varepsilon}$, as $n\rightarrow+\infty$.
Invoking \eqref{e4.2}, we get
\begin{equation}\label{e4.3}
\begin{aligned}
&\Upsilon'_{\varepsilon}(u_{\varepsilon})-\lambda_{\varepsilon}\Psi'(u_{\varepsilon})=0 \quad\mbox{in}\
~ (H^{2}(\mathbb{R}^{N}))^{*},\\
&\|\Upsilon'_{\varepsilon}(v_{n})-\lambda_{\varepsilon}\Psi'(v_{n})\|_{(H^{2}(\mathbb{R}^{N}))^{*}}\rightarrow0,\quad  \mbox{as}\ ~ n\rightarrow\infty.
\end{aligned}
\end{equation}
By straightforward calculations, we get
$$0>\tilde{d}_{\infty,T,c}>\lim\limits_{n\rightarrow+\infty}\Upsilon_{\varepsilon}(u_{n})
=\lim\limits_{n\rightarrow+\infty}(\Upsilon_{\varepsilon}(u_{n})-\frac{1}{2}\Upsilon'_{\varepsilon}(u_{n})u_{n}
+\frac{\lambda_{n}}{2}\|u_{n}\|_{2}^{2}+o_{n}(1))\geq\frac{1}{2}\lambda_{\varepsilon}c,$$
implying that
\begin{equation}\label{e4.4}
\lambda_{\varepsilon}\leq\frac{2\tilde{d}_{\infty,T,c}}{c}<0, \quad\mbox{for\ every}\ ~ \varepsilon\in(0,\varepsilon_{0}).
\end{equation}
By \eqref{e4.3}, and the above analysis, we obtain that
$$\|v_{n}\|_{2}^{2}+\int_{\mathbb{R}^{N}}V(\varepsilon x)|v_{n}|^{2}dx-\lambda_{\varepsilon}\|v_{n}\|_{2}^{2}
=\mu\|v_{n}\|_{q}^{q}+\|v_{n}\|_{2^{**}}^{2^{**}}+o_{n}(1), $$
which combined with \eqref{e4.4}, gives
\begin{equation}\label{e4.5}
\|\Delta v_{n}\|_{2}^{2}+\int_{\mathbb{R}^{N}}V(\varepsilon x)|v_{n}|^{2}dx-\frac{2\tilde{d}_{\infty,T,c}}{c}
\|v_{n}\|_{2}^{2}\leq\|v_{n}\|_{q}^{q}+\|v_{n}\|_{2^{**}}^{2^{**}}+o_{n}(1),
\end{equation}
so, together \eqref{e4.5} with the Sobolev inequality, we have
$$C_{1}\|v_{n}\|_{\varepsilon}^{2}\leq\mu\|v_{n}\|_{q}^{q}+\|v_{n}\|_{2^{**}}^{2^{**}}+o_{n}(1)
\leq \mu C_{2}\|v_{n}\|_{\varepsilon}^{q}+C_{3}\|v_{n}\|_{\varepsilon}^{2^{**}}+o_{n}(1).$$
Since $v_{n}\nrightarrow0$ in $H^{2}(\mathbb{R}^{N})$, there exists $C_{4}$ independent
of $\varepsilon$ such that $\|v_{n}\|_{\varepsilon}\geq C_{5}$. Moreover,
\begin{equation}\label{e4.6}
\liminf\limits_{n\rightarrow+\infty}(\mu\|v_{n}\|_{\varepsilon}^{q}+\|v_{n}\|_{\varepsilon}^{2^{**}})\geq C_{4}\quad
\hbox{ for some } C_{4}>0;
\end{equation}
thus, by  \eqref{e4.6} and the Gagliardo-Nirenberg inequality, we can conclude that there exists
$\beta>0$ independent of $\varepsilon\in(0,\varepsilon_{0})$ such that
$$\liminf\limits_{n\rightarrow+\infty}\|v_{n}\|_{2}\geq\beta.$$
This completes the proof of  Lemma \ref{lem4.4}.
\end{proof}
Hereafter, we will fix $$0<\rho_{0}<\min\left\{\tilde{d}_{\infty,T,c}-\tilde{d}_{0,T,c},\frac{\beta^{2}}{c}(
\tilde{d}_{\infty,T,c}-\tilde{d}_{0,T,c})\right\}.$$
\begin{lemma}\label{lem4.5}
The energy functional $\Upsilon_{\varepsilon,T}$ satisfies the $(PS)_{t}$ condition restricted to
$S_{c},$ provided that  $t<\tilde{d}_{0,T,c}+\rho_{0}.$
\end{lemma}
\begin{proof}
Let $\{u_{n}\}\subset S_{c}$ be a $(PS)_{t}$ sequence of $\Upsilon_{\varepsilon,T}$
restricted to $S_{c}$. Noting that $t<\tilde{d}_{\infty,,T,c}<0$, by Lemma \ref{lem4.2},
$\{u_{n}\}$ is bounded in $H^{2}(\mathbb{R}^{N})$. Taking the same argument as in the proof of Lemma \ref{lem4.3}, we can show that $u_{n}\rightharpoonup u_{\varepsilon}$
in $H^{2}(\mathbb{R}^{N})$ and $u_{\varepsilon}\not\equiv0$.
A straightforward computation gives that $v_{n}:=u_{n}-u_{\varepsilon}$ is a
$(PS)_{t'}$ sequence of $\Upsilon_{\varepsilon,T}$ restricted to $S_{c}$
and $t'<t$. If $v_{n}\nrightarrow0$ in $H^{2}(\mathbb{R}^{N}),$
then, by Lemma \ref{lem4.4},
$$\liminf\limits_{n\rightarrow+\infty}\|v_{n}\|_{2}\geq\beta.$$

Setting $b=\|u_{\varepsilon}\|_{2}^{2}$, $e_{n}=\|v_{n}\|_{2}^{2}$
and supposing that $\|v_{n}\|\rightarrow e$,
 we get $e\geq\beta^{2}>0$ and
$c=b+e$.
Since $v_{n}\rightharpoonup0,$  we can use a similar proof as for
\eqref{e4.1} to obtain that $\Upsilon_{\varepsilon,T}(v_{n})\geq\tilde{d}_{\infty,T,c}+o_{n}(1)$.
For every $e_{n}\in(0,c)$ for sufficiently large $n$, we have
\begin{equation}\label{e4.7}
t+o_{n}(1)=\Upsilon_{\varepsilon,T}(u_{n})\geq\Upsilon_{\varepsilon,T}(v_{n})
+\Upsilon_{\varepsilon,T}(u_{\varepsilon})\geq\tilde{d}_{\infty,T,e_{n}}+\tilde{d}_{0,T,b}+o_{n}(1).
\end{equation}
By Lemma \ref{lem3.7} and \eqref{e4.7}, we obtain that
$$\tilde{d}_{0,T,c}+\rho_{0}\geq t+o_{n}(1)\geq\frac{e_{n}}{c}\tilde{d}_{\infty,T,c}
+\frac{b}{c}\tilde{d}_{0,T,c}.$$
Taking the limits as $n\rightarrow+\infty$, we get
$$\rho_{0}\geq\frac{d}{c}(\tilde{d}_{\infty,T,c}-\tilde{d}_{0,T,c})\geq
\frac{\beta^{2}}{c}(\tilde{d}_{\infty,T,c}-\tilde{d}_{0,T,c}),$$
which contradicts $\rho_{0}<\frac{\beta}{c}(\tilde{d}_{\infty,T,c}-\tilde{d}_{0,T,c}),$
so we must have $u_{n}\rightarrow u_{\varepsilon}$ in $H^{2}(\mathbb{R}^{N})$.
This completes the proof of Lemma \ref{lem4.5}.
\end{proof}

\section{Multiplicity of Solutions}\label{S5}
In this section, we will establish the multiplicity of solutions of system
\eqref{e1.1}, invoking an argument from Alves  \cite{ac}.
To this end, fix $\tilde{\rho}$, $\tilde{r}>0$ such that
\begin{itemize}
  \item $\overline{B_{\tilde{\rho}}(k_{i})}\cap\overline{B_{\tilde{\rho}}(k_{j})}=\emptyset$  for
  $i,j$ and $k_{i}$, $k_{j}$ defined in $(V_{3})$;
  \item $\bigcup_{i=1}^{l}B_{\tilde{\rho}}(k_i)\subset B_{\tilde{r}}(0)$;
  \item $K_{\frac{\tilde{\rho}}{2}}=\bigcup_{i=1}^{l}\overline{B_{\frac{\tilde{\rho}}{2}}(k_{i})}.$
\end{itemize}
Define the function $\mathcal{Q}_{\varepsilon}:H^{2}(\mathbb{R}^{N})\setminus\{0\}\rightarrow
\mathbb{R}^{N}$ by
$$\mathcal{Q}_{\varepsilon}(u):=\frac{\int_{\mathbb{R}^{N}}\chi(\varepsilon x)|u|^{2}dx}{\int_{\mathbb{R}^{N}}|u|^{2}dx},$$
where $\chi:\mathbb{R}^{N}\rightarrow\mathbb{R}^{N}$ is given by
\begin{equation*}
\chi(x):=
\left\{
\begin{array}{lll}
x &\mbox{if}\ |x|\leq\tilde{r},\\
\tilde{r}\frac{x}{|x|} &\mbox{if}\ |x|>\tilde{r}.
\end{array}
\right.
\end{equation*}
The following two lemmas will be instrumental in generating $(PS)$ sequences
for $\Upsilon_{\varepsilon,T}$ on the constraints of $S_{c}$.

\begin{lemma}\label{lem4.6}
There exist $\varepsilon_{1}\in(0,\varepsilon_{0}]$ and $\rho_{1}\in(0,\rho_{1}]$
such that if $\varepsilon\in(0,\varepsilon_{1})$, $u\in S_{c}$, and $\Upsilon_{\varepsilon,T}
(u)\leq\tilde{d}_{0,T,c}+\rho_{1}$, then
$\mathcal{Q}_{\varepsilon}(u)\in K_{\frac{\tilde{\rho}}{2}}$.
\end{lemma}
\begin{proof}
Suppose, to the contrary, that the conclusion fails. Then, there exist sequences $\rho_{n}\rightarrow0$, $\varepsilon_{n}\rightarrow0$
and $\{u_{n}\}\subset S_{c}$ such that
\begin{equation}\label{e4.8}
\Upsilon_{\varepsilon_{n},T}(u_{n})\leq\tilde{d}_{0,T,c}+\rho_{n},\quad \mathcal{Q}_{\varepsilon_{n}}
(u_{n})\notin K_{\frac{\tilde{\rho}}{2}}.
\end{equation}
Consequently, we have
$$\tilde{d}_{0,T,c}\leq\Upsilon_{0,T}(u_{n})\leq\Upsilon_{\varepsilon_{n},T}(u_{n})
\leq\tilde{d}_{0,T,c}+\rho_{n}.$$
So, $\{u_{n}\}\subset S_{c}$ and $\Upsilon_{0,T}(u_{n})\rightarrow\tilde{d}_{0,T,c}$.
By Lemma \ref{lem3.9}, we need to consider two cases.

Case (i) $u_{n}\rightarrow u$ in $H_{\varepsilon}(\mathbb{R}^{N})$, for some $u\in S_{c}$.

Case (ii) There exists $\{y_{n}\}\subset \mathbb{R}^{N}$ with $|y_{n}|\rightarrow+\infty$
such that $v_{n}(x)=u_{n}(x+y_{n})$ converges to some $v\in S_{c}$ in $H^{2}(\mathbb{R}^{N})$.

Analysis of (i).
Applying the Lebesgue dominated convergence theorem, it follows that
$$\mathcal{Q}_{\varepsilon_{n}}(u_{n})=\frac{\int_{\mathbb{R}^{N}}\chi(\varepsilon_{n}x)|u_{n}|^{2}dx}
{\int_{\mathbb{R}^{N}}|u_{n}|^{2}dx}\rightarrow\frac{\int_{\mathbb{R}^{N}}\chi(0)|u|^{2}dx}
{\int_{\mathbb{R}^{N}}|u|^{2}dx}=k_{i}\in K_{\frac{\tilde{\rho}}{2}},$$
which contradicts $\mathcal{Q}_{\varepsilon_{n}}\notin K_{\frac{\tilde{\rho}}{2}}$.

Analysis of (ii).
We need to examine two possibilities:

(I) $ |\varepsilon_{n}y_{n}|\rightarrow+\infty$,

(II) $\varepsilon_{n}y_{n}\rightarrow y, ~ \mbox{for\ some}\ ~ y\in\mathbb{R}^{N}.$

If (I) holds, then, in view of $v_{n}\rightarrow v$ in $H^{2}(\mathbb{R}^{N})$, we obtain
\begin{eqnarray}\label{e4.9}
\Upsilon_{\varepsilon_{n},T}(u_{n}) &=&\nonumber
\frac{1}{2}\|\Delta v_{n}\|_{2}^{2}+\frac{1}{2}\int_{\mathbb{R}^{N}}V(\varepsilon_{n}x+\varepsilon_{n}y_{n})
|v_{n}|^{2}dx-\frac{\mu}{q}\|v_{n}\|_{q}^{q}\\
&&-\frac{\eta\|\Delta v_{n}\|_{2}}{2^{**}}\|\Delta v_{n}\|_{2^{**}}^{2^{**}}\rightarrow\Upsilon_{\infty,T}(v).
\end{eqnarray}
Since $\Upsilon_{\varepsilon_{n},T}(u_{n})\leq\tilde{d}_{0,T,c}+\rho_{n}$, we conclude that
$\tilde{d}_{\infty,T,c}\leq\Upsilon_{\infty,T}(v)\leq\tilde{d}_{0,T,c}$, which contradicts
$\tilde{d}_{\infty,T,c}>\tilde{d}_{0,T,c}.$

If (II) holds, similar to \eqref{e4.9}, $\Upsilon_{\varepsilon_{n},T}(u_{n})\rightarrow\Upsilon
_{y,T}(v)$, which combined with $\Upsilon_{\varepsilon_{n},T}(u_{n})\leq\tilde{d}_{0,T,c}+\rho_{n}$
implies that $\tilde{d}_{y,T,c}\leq\Upsilon_{y,T}(v)\leq\tilde{d}_{0,T,c}$. According to Corollary \ref{cor3.1},
it follows that $V(y)=V_{0}$ and $y=k_{i}$ for some $i=1,2,\cdots,l$. Consequently,
$$\mathcal{Q}_{\varepsilon_{n}}(u_{n})=\frac{\int_{\mathbb{R}^{N}}\chi(\varepsilon_{n}x+\varepsilon_{n}y_{n})|v_{n}|^{2}dx}
{\int_{\mathbb{R}^{N}}|v_{n}|^{2}dx}\rightarrow\frac{\int_{\mathbb{R}^{N}}\chi(y)|v|^{2}dx}
{\int_{\mathbb{R}^{N}}|v|^{2}dx}=k_{i}\in K_{\frac{\tilde{\rho}}{2}},$$
which implies that $\mathcal{Q}_{\varepsilon_{n}}(u_{n})\in K_{\frac{\tilde{\rho}}{2}}$ for sufficiently large $n$
and therefore contradicts with \eqref{e4.8}.
This completes the proof of Lemma \ref{lem4.6}.
\end{proof}
For convenience, we will use the following notations:
\begin{itemize}
  \item $\theta_{\varepsilon}^{i}:=\{u\in S_{c}:|\mathcal{Q}_{\varepsilon}(u)-k_{i}|<\tilde{\rho}\}$,
  $\partial\theta_{\varepsilon}^{i}:=\{u\in S_{c}:|\mathcal{Q}_{\varepsilon}(u)-k_{i}|=\tilde{\rho}\}$.
  \item $\beta_{\varepsilon}^{i}:=\inf\limits_{u\in\theta_{\varepsilon}^{i}}\Upsilon_{\varepsilon,T}(u)$,
  $\tilde{\beta}_{\varepsilon}^{i}:=\inf\limits_{u\in\partial\theta_{\varepsilon}^{i}}\Upsilon_{\varepsilon,T}(u)$.
\end{itemize}
\begin{lemma}\label{lem4.7}
There exists $\varepsilon_{2}\in(0,\varepsilon_{1}]$ such that
\begin{equation}\label{e4.10}
\beta_{\varepsilon}^{i}<\tilde{d}_{0,T,c}+\frac{\rho_{1}}{2} ~\mbox{and}\ ~
\beta_{\varepsilon}^{i}<\theta_{\varepsilon}^{i}, \quad \mbox{for\ every}\ ~ \varepsilon\in(0,\varepsilon_{2}).
\end{equation}
\end{lemma}
\begin{proof}
In what follows, let $u\in S_{c}$ satisfy $\Upsilon_{0,T}(u)=\tilde{d}_{0,T,c}$.
For $i\leq i\leq l$, we define the function $\hat{u}_{\varepsilon}^{i}:\mathbb{R}^{N}\rightarrow\mathbb{R}$ as
$$\hat{u}_{\varepsilon}^{i}(\cdot):=u\left(\cdot-\frac{k_{i}}{\varepsilon}\right).$$
Therefore, $\hat{u}_{\varepsilon}^{i}\in S_{c}$, for every $\varepsilon>0$ and
$i\leq i\leq k$. By a straightforward change of variable,
it can be shown that
$$\Upsilon_{\varepsilon,T}(\hat{u}_{\varepsilon}^{i})=\frac{1}{2}\|\Delta u\|_{2}^{2}+\frac{1}{2}\int_{\mathbb{R}^{N}}V(\varepsilon_{n}x+k_{i})
|u|^{2}dx-\frac{\mu}{q}\|u\|_{q}^{q}-\frac{\eta\|\Delta u\|_{2}}{2^{**}}\|\Delta u\|_{2^{**}}^{2^{**}}$$
and
\begin{equation}\label{e4.11}
\lim\limits_{\varepsilon\rightarrow0^{+}}\Upsilon_{\varepsilon,T}(\hat{u}_{\varepsilon}^{i})=
\Upsilon_{k_{i},T}(u)=\Upsilon_{0,T}=\tilde{d}_{0,T,c}.
\end{equation}
Note that, as $\varepsilon\rightarrow0^{+}$, $\mathcal{Q}_{\varepsilon}(\hat{u}_{\varepsilon}^{i})
\rightarrow k_{i}$, this implies that $\hat{u}_{\varepsilon}^{i}\in\mathcal{Q}_{\varepsilon}^{i}$
when $\varepsilon$ is sufficiently small. According to \eqref{e4.11}, there exists $\varepsilon_{2}\in(0,\varepsilon_{1}]$
such that
$$\beta_{\varepsilon}^{i}<\tilde{d}_{0,T,c}+\frac{\rho_{1}}{2}, \quad \mbox{for\ every}\ ~
\varepsilon\in(0,\varepsilon_{2}),$$
thus proving the first inequality in \eqref{e4.10}.

For any $v\in\partial\beta_{\varepsilon}^{i}$, we can conclude that
$\mathcal{Q}_{\varepsilon}(v)\notin K_{\frac{\tilde{\rho}}{2}}$. Thus,
by Lemma \ref{lem4.6}, one has
$$\Upsilon_{\varepsilon,T}(v)>\tilde{d}_{0,T,c}+\rho_{1},\quad \mbox{for\ every} ~ v\in\partial\theta_{\varepsilon}^{i}
~\mbox{and}\ ~ \varepsilon\in(0,\varepsilon_{2}).$$
This implies that $$\tilde{\beta}_{\varepsilon}^{i}:=\inf\limits_
{u\in\partial\theta_{\varepsilon}^{i}}\Upsilon_{\varepsilon,T}(v)\geq\tilde{d}_{0,T,c}
+\rho_{1},\quad \text{for every } \varepsilon\in(0,\varepsilon_{2}).$$ Consequently,
$$\beta_{\varepsilon}^{i}<\beta_{\varepsilon}, ~\mbox{for\ every} ~ \varepsilon\in(0,\varepsilon_{2}).$$
This completes the proof of Lemma \ref{lem4.7}.
\end{proof}

\begin{proposition}\label{pro4.1}
Let $\varepsilon\in(0,\tilde{\varepsilon})$ be fixed. Then decreasing $\tilde{\varepsilon}>0$
if necessary, $\Upsilon_{\varepsilon}|_{S_{c}}$ has at least $k$ different nontrivial critical points.
\end{proposition}
\begin{proof}
Let $\varepsilon\in(0,\tilde{\varepsilon}),$ where $\tilde{\varepsilon}:=\varepsilon_{2}$
as determined in Lemma \ref{lem4.7}, and choose $i\in\{1,2,\cdots,l\}.$  Then, the
Ekeland variational principle implies that there exists  a sequence $\{u_{n}^{i}\}\subset\theta_{\varepsilon}^{i}$
satisfying
$$\Upsilon_{\varepsilon,T}(u_{n}^{i})\rightarrow\beta_{\varepsilon}^{i}~\mbox{and}\ ~
\|\Upsilon_{\varepsilon,T}|_{S'_{c}}(u_{n}^{i})\|_{(H^{2}(\mathbb{R}^{N}))^{*}}\rightarrow0, \quad
\mbox{as}\ ~n\rightarrow\infty.$$
In other words, $\{u_{n}^{i}\}$ is a $(PS)_{_{\varepsilon}^{i}}$ sequence for
$\Upsilon_{\varepsilon,T}$ when restricted on $S_{c}$. Because $\beta_{\varepsilon}^{i}
<\tilde{d}_{0,T,c}+\rho_{0}$, Lemma \ref{lem3.8} guarantees the existence of
$u^{i}$ with $u_{n}^{i}\rightarrow u^{i}$ in $H^{2}(\mathbb{R}^{N})$. Therefore,
$$u^{i}\in\theta_{\varepsilon}^{i},\quad\Upsilon_{\varepsilon,T}(u^{i})=\beta_{\varepsilon}^{i},
\quad \mbox{and}\quad \Upsilon_{\varepsilon,T}|_{S'_{c}}(u^{i})=0.$$
For
$$\mathcal{Q}_{\varepsilon}(u^{i})\in\overline{B_{\tilde{\rho}}(k_{i})},\quad
\mathcal{Q}_{\varepsilon}(u^{j})\in\overline{B_{\tilde{\rho}}(k_{j})}\quad\mbox{and}\quad
\overline{B_{\tilde{\rho}}(k_{j})}\cap\overline{B_{\tilde{\rho}}(k_{i})}
=\emptyset,
 \quad
\mbox{for}\quad   i\neq j,$$
we get $u^{i}\not\equiv u^{j}$ for $i\neq j$, where $i\leq i,j\leq l$. This argument shows that
$\Upsilon_{\varepsilon,T}$ possesses at least $l$ nontrivial critical points for
every $\varepsilon\in(0,\tilde{\varepsilon})$. Using Lemma \ref{lem4.2} and the fact
that $\Upsilon_{\varepsilon,T}(u^{i})<0$ for every  $i=1,2,\cdots,l$, it is obvious
that $u^{i}$ are in fact the critical points of $\Upsilon_{\varepsilon}$ restricted on
$S_{c}$ with $\Upsilon_{\varepsilon}(u^{i})=\beta_{\varepsilon}^{i}<0$ and $\Upsilon'_{\varepsilon}
u^{i}=\lambda_{i}c$. Therefore, we can deduce that
$$\frac{1}{2}\lambda_{i}c=\Upsilon_{\varepsilon}(u^{i})+\left(\frac{1}{q}-\frac{1}{2}\right)\mu\|u^{i}\|_{q}^{q}+
\left(\frac{1}{2^{**}}-\frac{1}{2}\right)\|u^{i}\|_{2^{**}}^{2^{**}};$$
thus, $\lambda_{i}<0$,  for every $i=1,2,\cdots,l$.
This completes the proof of Proposition \ref{pro4.1}.
\end{proof}
In order to prove the concentrating behavior of positive normalized solutions
of system \eqref{e1.1}, by the scaling $v(x)=u(\varepsilon x)$, we have to
consider the following problem which is equivalent to system \eqref{e1.1}
\begin{equation}\label{e4.12}
\Delta^{2}v+V(\varepsilon x)v=\lambda v+\mu |v|^{q-2}v+|v|^{2^{**}-2}v \quad\mbox{in}\ \mathbb{R}^{N}.
\end{equation}
In order words, if the couple $(v,\lambda)$ is a weak solution of problem \eqref{e4.12},
then $(v,\lambda)$ is a solution of problem \eqref{e1.1}. By Proposition \ref{pro4.1},
there are $k$ couples of $(u_{\varepsilon}^{i},\lambda_{\varepsilon}^{i})\in H^{2}(\mathbb{R}^{N})
\times\mathbb{R}$ such that
$$v_{\varepsilon}^{i}\in\theta_{\varepsilon}^{i},\Upsilon_{\varepsilon}(v_{\varepsilon}^{i})=
\beta_{\varepsilon}^{i} ~\mbox{and}\ ~\Upsilon'_{\varepsilon}(v_{\varepsilon}^{i})-
\theta_{\varepsilon}^{i}\Psi'(v_{\varepsilon}^{i})=0\quad \text{in}\ H_{\varepsilon}(\mathbb{R}^{N}),$$
where $i\in\{1,2,\cdots,k\}$, $v_{\varepsilon}^{i}(x)>0$ for every $x\in\mathbb{R}^{N}$ and $\lambda^{i}<0$.
\begin{lemma}\label{lem4.8}
Let $\varepsilon\in(0,\tilde{\varepsilon})$ be fixed. Decreasing $\tilde{\varepsilon}>0,$ there exist $y_{\varepsilon}^{i}\in\mathbb{R}^{N}$, $R_{0}^{i}>0$ and
$\beta_{0}^{i}$ such that
$$\int_{B_{R_{0}}(y_{\varepsilon}^{i})}|v_{\varepsilon}^{i}|^{2}dx\geq\beta_{0}^{i}$$
for every $i \in \{1,2,\cdots,k\}$. Furthermore, the sequence $\{\varepsilon y_{\varepsilon}^{i}\}_{i}$
is bounded and, passing to a subsequence if necessary, $\varepsilon y_{\varepsilon}^{i}\rightarrow x^{i},$
as $\varepsilon\rightarrow0^{+}$.
\end{lemma}
\begin{proof}
Arguing by contradiction, suppose that there exists a sequence $\{\varepsilon_{n}\}_{n}$ with $\varepsilon_{n}\rightarrow0^{+}$
such that
$$\lim\limits_{n\rightarrow\infty}\sup_{y\in\mathbb{R}^{N}}\int_{B_{r}(y)}|v_{\varepsilon_{n}^{i}}|^{2}dx=0.$$
By the Lions' vanishing lemma in Molica Bisci et al. \cite{bv}, we can deduce that
\begin{equation*}
v_{\varepsilon_{n}}^{i}\rightarrow0 ~\mbox{in}\ L^{\nu}(\mathbb{R}^{N}),
\quad\mbox{for\ all}\ \nu\in(2,2^{**}).
\end{equation*}
Therefore, $\lim\limits_{n\rightarrow\infty}\Upsilon_{\varepsilon}(v_{\varepsilon_{n}}^{i})\geq0$
which contradicts the fact that
\begin{equation}\label{e4.14}
\lim_{n\rightarrow\infty}\Upsilon_{\varepsilon}(v_{\varepsilon_{n}}^{i})=
\lim\limits_{n\rightarrow\infty}\beta_{\varepsilon_{n}}^{i}\leq\tilde{d}_{0,T,c}+\rho<0.
\end{equation}
Therefore, we define $\tilde{v}_{\varepsilon}^{i}(\cdot)=v_{\varepsilon}^{i}
(\cdot+y_{\varepsilon}^{i})$ and $\{\tilde{v}_{\varepsilon}^{i}\}_{i}$ is bounded for every  $\varepsilon\in(0,\tilde{\varepsilon})$. Therefore, there exists $\tilde{v}\in H_{\varepsilon}
(\mathbb{R}^{N})\setminus\{0\}$ such that $\tilde{v}_{\varepsilon}^{i}\rightharpoonup \tilde{v}^{i}$
in $H_{\varepsilon}(\mathbb{R}^{N}),$ as $\varepsilon\rightarrow0^{+}$ along a subsequence.
Since $\{\tilde{v}_{\varepsilon}^{i}\}_{i}\subset S_{c}$ and
$$\Upsilon_{\varepsilon}(v_{\varepsilon}^{i})\geq\Upsilon_{0}(v_{\varepsilon}^{i})
=\Upsilon_{0}(\tilde{v}_{\varepsilon}^{i})\geq\tilde{d}_{0,T,c},$$
this jointly with \eqref{e4.11} yields that $\lim\limits_{\varepsilon\rightarrow0^{+}}\Upsilon_{0}(\tilde{v}_{\varepsilon}^{i})
=\Upsilon_{0,T,c}$. By Lemma \ref{lem3.9}, we see that $\tilde{v}_{\varepsilon}^{i}\rightarrow\tilde{v}$ in
$H_{\varepsilon}(\mathbb{R}^{N})$,  as $\varepsilon\rightarrow0^{+}$.
Suppose that $\{\varepsilon y_{\varepsilon}^{i}\}_{i}$ is unbounded with respect to
$\varepsilon\in(0,\tilde{\varepsilon})$, then there exists a subsequence
$\{\varepsilon_{n} y_{\varepsilon_{n}}^{i}\}_{i}$ such that $\varepsilon_{n} y_{\varepsilon_{n}}^{i}\rightarrow+
\infty$ as $n\rightarrow\infty$. Exploiting $\tilde{\varepsilon}_{\varepsilon_{n}}^{i}\rightarrow\tilde{\varepsilon}$
in $H_{\varepsilon}(\mathbb{R}^{N})$, one has
\begin{eqnarray*}
\Upsilon_{\varepsilon_{n}}(v_{\varepsilon_{n}}^{i}) &=&\nonumber
\frac{1}{2}\|\Delta v_{\varepsilon_{n}}^{i}\|_{2}^{2}+\frac{1}{2}\int_{\mathbb{R}^{N}}V(\varepsilon_{n}x)
|v_{\varepsilon_{n}}^{i}|^{2}dx-\int_{\mathbb{R}^{N}}G(v_{\varepsilon_{n}}^{i})dx\\
&=&\frac{1}{2}\|\Delta \tilde{v}_{\varepsilon_{n}}^{i}\|_{2}^{2}+\frac{1}{2}\int_{\mathbb{R}^{N}}V(\varepsilon_{n}x+\varepsilon_{n}y_{n})
|\tilde{v}_{\varepsilon_{n}}^{i}|^{2}dx-\int_{\mathbb{R}^{N}}G(\tilde{v}_{\varepsilon_{n}}^{i})dx
\rightarrow\Upsilon_{\infty}(\tilde{v}),
\end{eqnarray*}
so invoking \eqref{e4.14}, we obtain the following inequality:
$$\tilde{d}_{0,T,c}+\frac{\rho_{1}}{2}\geq\Upsilon_{\infty}(\tilde{v})\geq\tilde{d}_{\infty,T,c}.$$
Due to Lemma \ref{lem3.7}, this is in contradiction with the definition of $\rho$
given in Lemma \ref{lem4.6}. Therefore, up to a subsequence, $\varepsilon y_{\varepsilon}^{i}
\rightarrow x_{0}^{i}$ in $\mathbb{R}^{N}$ as $\varepsilon\rightarrow0^{+}$.

In the sequel, we will verify that $x_{0}^{i}=x^{i}$. Using a
similar argument as in the case $(ii)$ in the proof of Lemma \ref{lem4.6}, we can infer that
$V(x_{0}^{i})=V_{0}$. Recalling $\{v_{\varepsilon}^{i}\}\in\theta_{\varepsilon}^{i}$,
we know that 
$$\lim\limits_{n\rightarrow\infty}\mathcal{Q}_{\varepsilon_{n}}(v_{\varepsilon}^{i})=x_{0}^{i}.$$
Moreover, we have that $|x^{i}-x_{0}^{i}|\leq\rho_{0}$. Hence, we obtain that
$x_{0}^{i}=x^{i}$.
This completes the proof of  Lemma \ref{lem4.8}.
\end{proof}
\begin{lemma}\label{lem4.9}
Let $\varepsilon\in(0,\tilde{\varepsilon})$ be fixed. Decreasing $\varepsilon^{*}>0$ if necessary, there exists $v_{\varepsilon}^{i}$ possessing a maximum $\mu_{\varepsilon}^{i},$ satisfying
$V(\varepsilon \mu_{\varepsilon}^{i})\rightarrow V(x^{i}),$ as $\varepsilon\rightarrow0^{+},$
for every $i\in\{1,2,\cdots,k\}$.
\end{lemma}
\begin{proof}
Since $\tilde{v}_{\varepsilon}^{i}(\cdot)=v_{\varepsilon}^{i}(\cdot+y_{\varepsilon}^{i})$,
the definition of $v_{\varepsilon}^{i}$ yields that a couple of weak solutions $(\tilde{v}_{\varepsilon}^{i},
\lambda_{\varepsilon}^{i})$ solving the following problem:
\begin{equation}\label{e4.15}
\left\{
\begin{array}{lll}
\Delta^{2}\tilde{v}_{\varepsilon}^{i}+V(\varepsilon x+\varepsilon x_{\varepsilon}^{i})\tilde{v}_{\varepsilon}^{i}=\lambda_{\varepsilon}^{i} \tilde{v}_{\varepsilon}^{i}+\mu |\tilde{v}_{\varepsilon}^{i}|^{q-2}\tilde{v}_{\varepsilon}^{i}+|\tilde{v}_{\varepsilon}^{i}|^{2^{**}-2}\tilde{v}_{\varepsilon}^{i} &\mbox{in}\ \mathbb{R}^{N},\\
\int_{\mathbb{R}^{N}}|\tilde{v}_{\varepsilon}^{i}|^{2}dx=c^{2} &\mbox{in}\ \mathbb{R}^{N}.
\end{array}
\right.
\end{equation}
Arguing as in the proof of Proposition \ref{pro4.1} and Lemma \ref{lem4.8}, we know that
$\tilde{v}_{\varepsilon}^{i}\rightarrow\tilde{v}_{\varepsilon}^{i}$ in $H_{\varepsilon}(\mathbb{R}^{N})$,
$\lambda_{\varepsilon}^{i}\rightarrow\lambda^{i}$ in $\mathbb{R}^{N}$ and $\varepsilon x_{\varepsilon}^{i}
\rightarrow x^{i}$ in $\mathbb{R}^{N},$ as $\varepsilon\rightarrow0^{+}$. So, using \eqref{e4.15},
$(\tilde{v}^{i},\lambda^{i})$ is a nontrivial solution to
$$\Delta^{2}\tilde{v}^{i}+V_{0}\tilde{v}^{i}=\lambda^{i} \tilde{v}^{i}+\mu |\tilde{v}^{i}|^{q-2}\tilde{v}^{i}+|\tilde{v}^{i}|^{2^{**}-2}\tilde{v}^{i} \quad\mbox{in}\ \mathbb{R}^{N}.$$
We will divide the proof into two steps.

\textbf{Step 1.}
We verify that $\|\tilde{v}^{i}_{\varepsilon}\|_{\infty}\geq\rho^{i}$ and
\begin{equation}\label{e4.16}
\lim_{|x|\rightarrow\infty}\tilde{v}_{\varepsilon}^{i}(x)=0,
\end{equation}
where $\rho^{i}>0$ is independent of $\varepsilon\in(0,\tilde{\varepsilon})$

We prove the first assertion. Assuming to the contrary, we get that $|\tilde{v}_{\varepsilon}^{i}|_{\infty}
\rightarrow0,$ as $\varepsilon\rightarrow0^{+}$ in the sense of a subsequence. Then,
it is easy to obtain that $\tilde{v}_{\varepsilon}^{i}\rightarrow0$ in $H_{\varepsilon}
(\mathbb{R}^{N}),$ which on the other hand  is known to be impossible. Therefore,
$$\|\tilde{v}_{\varepsilon}^{i}\|_{\infty}\geq\rho^{i},$$
where $\rho^{i}>0$ is independent of $\varepsilon\in(0,\tilde{\varepsilon})$.

In the sequel, we will verify that \eqref{e4.16} holds. For any $L>0$ and $\beta>1$,
we consider the function $\varpi(m)=m(\min\{m,L\})^{2\beta-1}$ and
$$\varpi(\tilde{v}_{\varepsilon}^{i})=\varpi_{\varepsilon,L}(\tilde{v}_{\varepsilon}^{i})
=\tilde{v}^{i}_{\varepsilon}\tilde{v}^{2(\beta-1)}_{\varepsilon,L}\in H_{\varepsilon}
(\mathbb{R}^{N}),\quad\tilde{v}_{\varepsilon,L}=\{\tilde{v}_{\varepsilon},L\}.$$
Set
$$\mathcal{A}(m)=\frac{|m|^{2}}{2}\quad\mbox{and}\ ~ \mathcal{B}=
\int_{0}^{m}(\varpi'(\eta))^{\frac{1}{2}}d\eta.$$
By the similar argument as in Ambrosio and Isernia \cite{am}, we have
\begin{equation}\label{e4.17}
\mathcal{A}'(x-y)(\varpi(x)-\varpi(y))\geq|\mathcal{B}(x)-
\mathcal{B}(y)|^{2},\quad\mbox{for\ every}\ ~x,y\in\mathbb{R}.
\end{equation}
From \eqref{e4.17}, we get
\begin{equation*}
|\mathcal{B}(\tilde{v}_{\varepsilon}^{i}(x))-\mathcal{B}(\tilde{v}_{\varepsilon}^{i}(x))|^{2}
\leq(\tilde{v}^{i}_{\varepsilon}(x)-\tilde{v}^{i}_{\varepsilon}(y))
((\tilde{v}^{i}_{\varepsilon}\tilde{v}^{2(\beta-1)}_{\varepsilon,L})(x)-
(\tilde{v}^{i}_{\varepsilon}\tilde{v}^{2(\beta-1)}_{\varepsilon,L})(y)).
\end{equation*}
Therefore, taking $\varpi(\tilde{v}_{\varepsilon}^{i})=
\tilde{v}^{i}_{\varepsilon}\tilde{v}^{2\beta}_{\varepsilon,L}$ as a test function
in \eqref{e4.15}, we have
\begin{align}\label{e4.19}
&\|\Delta(\mathcal{B}(\tilde{v}^{i}_{\varepsilon}))\|_{2}^{2}+\int_{\mathbb{R}^{N}}V_{\varepsilon}(x)
|\tilde{v}^{i}_{\varepsilon}|^{2}\tilde{v}^{2(\beta-1)}_{\varepsilon,L}dx \notag\\ &\leq
\int_{\mathbb{R}^{N}}\Delta(\tilde{v}^{i}_{\varepsilon}(x)-\tilde{v}^{i}_{\varepsilon}(y))
\Delta((\tilde{v}^{i}_{\varepsilon}\tilde{v}^{2\beta}_{\varepsilon,L})(x)-
(\tilde{v}^{i}_{\varepsilon}\tilde{v}^{2(\beta-1)}_{\varepsilon,L})(y))dx
\notag\\
&\quad+\int_{\mathbb{R}^{N}}V_{\varepsilon}(x)|\tilde{v}^{i}_{\varepsilon}|^{2}\tilde{v}^{2(\beta-1)}_{\varepsilon,L}dx \notag\\
&=\lambda_{\varepsilon}^{i}\int_{\mathbb{R}^{N}}|\tilde{v}^{i}_{\varepsilon}|^{2}\tilde{v}^{2(\beta-1)}_{\varepsilon,L}dx
+\mu\int_{\mathbb{R}^{N}}|\tilde{v}^{i}_{\varepsilon}|^{q}\tilde{v}^{2(\beta-1)}_{\varepsilon,L}dx+
\int_{\mathbb{R}^{N}}|\tilde{v}^{i}_{\varepsilon}|^{2^{**}}\tilde{v}^{2(\beta-1)}_{\varepsilon,L}dx,
\end{align}
where $V_{\varepsilon}(x)=V(\varepsilon x+\varepsilon x_{\varepsilon}^{i})$. Using \eqref{e4.17},
we have $\tilde{v}^{i}_{\varepsilon}v^{\beta-1}_{\varepsilon,L}\geq|\mathcal{B}(\tilde{v}_{\varepsilon}^{i})|$.
Since $\mathcal{B}(\tilde{v}_{\varepsilon}^{i})|\geq\frac{1}{\beta}
\tilde{v}^{i}_{\varepsilon}v^{\beta-1}_{\varepsilon,L}$ and the embedding is continuous
$H_{\varepsilon}(\mathbb{R}^{N})\hookrightarrow L^{t}(\mathbb{R}^{N})$,
for every $t\in[2,2^{**}]$, there exists a suitable constant $\tilde{S}_{1}>0$
such that
\begin{equation}\label{e4.20}
\|\Delta(\mathcal{B}(\tilde{v}^{i}_{\varepsilon}))\|_{2}^{2}\geq \tilde{S}_{1}
\|\Delta(\mathcal{B}(\tilde{v}^{i}_{\varepsilon}))\|_{2^{**}}^{2}\geq\frac{1}{\beta^{2}}
\tilde{S}_{1}\|\tilde{v}_{\varepsilon}^{i}\tilde{v}_{\varepsilon,L}^{\beta-1}\|^{2}_{2^{**}}.
\end{equation}
By \eqref{e4.19}, \eqref{e4.20}, we obtain that
\begin{align}\label{e4.21}
& \frac{1}{\beta^{2}}\tilde{S}_{1}\|\tilde{v}_{\varepsilon}^{i}\tilde{v}_{\varepsilon,L}^{\beta-1}\|^{2}_{2^{**}}
  +\int_{\mathbb{R}^{N}}V_{\varepsilon}(x)|\tilde{v}_{\varepsilon}^{i}\tilde{v}_{\varepsilon,L}^{\beta-1}|^{2}dx\notag \\
&\leq \lambda_{\varepsilon}^{i}\int_{\mathbb{R}^{N}}|\tilde{v}_{\varepsilon}^{i}\tilde{v}_{\varepsilon,L}^{\beta-1}|
^{2}dx+\mu\int_{\mathbb{R}^{N}}|\tilde{v}_{\varepsilon}^{i}\tilde{v}_{\varepsilon,L}^{\beta-1}|^{q}dx
+\int_{\mathbb{R}^{N}}|\tilde{v}_{\varepsilon}^{i}\tilde{v}_{\varepsilon,L}^{\beta-1}|^{2^{**}}dx.
\end{align}
By the proof of Lemma \ref{lem4.4}, we obtain $\lambda_{\varepsilon}^{i}\leq(\lambda^{*})^{i}
<0,$ for every $\varepsilon\in(0,\tilde{\varepsilon})$. Choosing $0<\mu<\frac{V_{0}}{2}$ and using
\eqref{e4.19} and \eqref{e4.21}, we deduce that
\begin{equation}\label{e4.22}
|\tilde{v}_{\varepsilon}^{i}\tilde{v}^{\beta-1}_{\varepsilon,L}|_{2^{**}}^{2}
\leq C\beta^{2}\int_{\mathbb{R}^{N}}|\tilde{v}_{\varepsilon}|^{2^{**}}\tilde{v}_{\varepsilon,L}^{2(\beta-1)}dx.
\end{equation}
Now, we take $\beta=\frac{2^{**}}{2}$ and fix $R>0$. Noting that
$0\leq\tilde{v}_{\varepsilon,L}\leq\tilde{v}_{\varepsilon}$, we can infer that
\begin{align}\label{e4.23}
&\int_{\mathbb{R}^{N}}\tilde{v}^{2^{**}}_{\varepsilon}\tilde{v}^{R(\beta-1)}_{\varepsilon,L}dx \notag\\
&=\int_{\mathbb{R}^{N}}\tilde{v}^{2^{**}-2}_{\varepsilon}\tilde{v}^{2}_{\varepsilon}\tilde{v}^{2^{**}-2}_{\varepsilon,L}dx \notag\\
&=\int_{\mathbb{R}^{N}}\tilde{v}^{2^{**}-2}_{\varepsilon}(\tilde{v}_{\varepsilon}\tilde{v}
^{\frac{2^{**}-2}{2}}_{\varepsilon,L})^{2}dx \notag\\
&\leq\int_{\{\tilde{v}_{\varepsilon}<R\}}R^{2^{**}-2}\tilde{v}_{\varepsilon}^{2^{**}}dx+\int_{\{\tilde{v}_{\varepsilon}>R\}}
\tilde{v}^{2^{**}-2}_{\varepsilon}(\tilde{v}_{\varepsilon}\tilde{v}^{\frac{2^{**}-2}{2}}_{\varepsilon,L})^{2}dx \notag \\
&\leq\int_{\{\tilde{v}_{\varepsilon}<R\}}R^{2^{**}-2}\tilde{v}_{\varepsilon}^{2^{**}}dx+\left(\int_{\{\tilde{v}_{\varepsilon}>R\}}
\tilde{v}_{\varepsilon}^{2^{**}}dx\right)^{\frac{2^{**}-2}{2}}\left(\int_{\mathbb{R}^{N}}(\tilde{v}_{\varepsilon}\tilde{v}_{\varepsilon,L}
^{\frac{2^{**}-2}{2}})^{2^{**}}dx\right)^{\frac{2}{2^{**}}}.
\end{align}
Since $\{\tilde{v}_{\varepsilon}\}_{n\in\mathbb{N}}$ is bounded in $L^{2^{**}}
(\mathbb{R}^{N})$,  we know  that for every sufficiently large $R$,
\begin{equation*}
\int_{\{\tilde{v}_{\varepsilon}<R\}}R^{2^{**}-2}\tilde{v}_{\varepsilon}^{2^{**}}dx\leq
\frac{1}{2C\beta^{2}}.
\end{equation*}
From \eqref{e4.22} and \eqref{e4.23}, we get
$$\left(\int_{\mathbb{R}^{N}}\left(\tilde{v}_{\varepsilon}\tilde{v}_{\varepsilon,L}
^{\frac{2^{**}-2}{2}}\right)^{2^{**}}dx\right)^{\frac{2}{2^{**}}}\leq C\beta^{2}\int_{\mathbb{R}^{N}}R^{2^{**}-2}
\tilde{v}_{\varepsilon}^{2^{**}}dx<\infty,$$
and taking the limit as $L\rightarrow\infty$, we obtain $\tilde{v}_{\varepsilon}
\in L^{\frac{(2^{**})^{2}}{2}}(\mathbb{R}^{N})$. Now, using $0\leq\tilde{v}_{\varepsilon,L}\leq\tilde{v}_{\varepsilon}$
and passing to the limit as $L\rightarrow\infty$ in \eqref{e4.5}, we have
$$|\tilde{v}_{\varepsilon}|^{2\beta}_{2^{**}\beta}\leq C\beta^{2}\int_{\mathbb{R}^{N}}
\tilde{v}_{\varepsilon}^{2^{**}+2(\beta-1)}dx,$$
from which we deduce that
$$\left(\int_{\mathbb{R}^{N}}\tilde{v}_{\varepsilon}^{2^{**}\beta}dx\right)^{\frac{1}{2^{**}(\beta-1)}}
\leq (C\beta)^{\frac{1}{\beta-1}}\left(\int_{\mathbb{R}^{N}}\tilde{v}_{\varepsilon}^{2^{**}+2(\beta-1)}dx\right)
^{\frac{1}{2(\beta-1)}}.$$
For $t\geq1$, we define $\beta_{t+1}$ inductively so that
$2^{**}+2(\beta_{t+1}-1)=2^{**}\beta_{t}$ and $\beta_{1}=\frac{2^{**}}{2}$.
Then, we have
$$\left(\int_{\mathbb{R}^{N}}\tilde{v}_{\varepsilon}^{2^{**}\beta_{t+1}}dx\right)^{\frac{1}{2^{**}(\beta_{t+1}-1)}}
\leq (C\beta_{t+1})^{\frac{1}{\beta_{t+1}-1}}\left(\int_{\mathbb{R}^{N}}\tilde{v}_{\varepsilon}^{2^{**}\beta_{t}}dx\right)
^{\frac{1}{2(\beta_{t}-1)}}.$$
Let us define
$$D_{t}=\left(\int_{\mathbb{R}^{N}}\tilde{v}_{\varepsilon}^{2^{**}\beta_{t}}dx\right)^{\frac{1}{2^{**}(\beta_{t}-1)}}.$$
Using a standard iteration argument, we can find $C_{0}>0$ independent of $t$ such that
$$D_{t+1}\leq\prod\limits_{k=1}^{t}(C\beta_{k+1})^{\frac{1}{k_{k+1}-1}}D_{1}\leq C_{0}D_{1}.$$
Taking the limit as $t\rightarrow\infty$, we get
$$|v_{n}|_{\infty}\leq C,\quad
\hbox{ for every } n.$$
Since $\tilde{v}_{\varepsilon}^{i}\rightarrow\tilde{v}^{i}$,
we can deduce that \eqref{e4.16} is true.

\textbf{Step 2. }
We verify that $\tilde{v}_{\varepsilon}^{i}$ possesses a maximum $\mu_{\varepsilon}^{i}$
satisfying $V(\varepsilon)\mu_{\varepsilon}^{i}\rightarrow V(x^{i})$. In the
sequel, let $\varrho_{\varepsilon}^{i}$ be a maximum of $\tilde{v}_{\varepsilon}^{i}.$
We have  $|\tilde{v}_{\varepsilon}^{i}(\varrho_{\varepsilon}^{i})|_{\infty}\geq\rho^{i}$.
Since $\lim_{|x|\rightarrow R}\bar{v}_{\varepsilon}^{j}=0$ uniformly in
$\varepsilon$,  there exists $\mathcal{R}_{0}^{i}>0$ independent of $\varepsilon$ such that
$|\varrho_{\varepsilon}^{i}|\leq\mathcal{R}_{0}^{i}$. Recalling $\tilde{v}_{\varepsilon}^{i}(\cdot)
=\tilde{v}_{\varepsilon}^{i}(\cdot+y_{\varepsilon}^{i})$, we get that
$y_{\varepsilon}^{i}+\varrho_{\varepsilon}^{i}$ is a maximum of $v_{\varepsilon}^{i}$.
Define $\mu_{\varepsilon}^{i}=y_{\varepsilon}^{i}+\varrho_{\varepsilon}^{i}.$
Invoking Lemma \ref{lem4.8} and $|\varrho_{\varepsilon}^{j}|\leq\mathcal{R}_{0}^{i}$,
one has $\varepsilon\mu_{\varepsilon}^{i}\rightarrow x^{i}$ as
$\varepsilon\rightarrow0^{+}$, and hence, $V(\varepsilon\mu_{\varepsilon}^{i})
\rightarrow V(x^{i})$ by the continuity of $V$.
This completes the proof of  Lemma \ref{lem4.9}.
\end{proof}

We can now give the proof of the main result of this paper.\\

\textbf{Proof of Theorem \ref{the1.1}} By Proposition \ref{pro4.1} and Lemma \ref{lem4.9},
we obtain that system \eqref{e1.1} has at least $k$ different couples of
solutions $(v_{\varepsilon}^{i},\lambda_{\varepsilon}^{i})\in H^{2}
(\mathbb{R}^{N})\times\mathbb{R}$ with $v_{\varepsilon}^{i}(x)>0$,
for every  $x\in\mathbb{R}^{N}$. Moreover, $\lambda_{\varepsilon}^{i}<0$,
where $i\in\{1,2,\cdots,k\}$. Let $u_{\varepsilon}^{i}(\cdot)=v_{\varepsilon}^{i}(\cdot\setminus\varepsilon)$
and $z_{\varepsilon}^{i}=\varepsilon\mu_{\varepsilon}^{i}$ for $i\in\{1,2,\cdots,k\}$;
then, $(u_{\varepsilon}^{i},\lambda_{\varepsilon}^{i})$ is the desired
solution for $i\in\{1,2,\cdots,k\}$.

This completes the proof of Theorem \ref{the1.1}.\qed
\section{Epilogue}\label{S6}

On concluding the paper, we summarize the main features of the main result.\\
\begin{itemize}
\item[($a$)]  Compared with the results of  Alves and Thin \cite{AT}, who assumed that the nonlinearity $f(u)$
satisfies $L^2$-subcritical growth, 
the present paper addresses
the nonlinearity $f(u)=\mu|u|^{q-2}u+|u|^{2^\ast\ast-2}u$ with $L^2$-subcritical growth and Sobolev critical growth.  
Thus, to some extent, our main result is a generalization of Alves and Thin \cite{AT}.

\item[($b$)]
 Chen and  Chen \cite{cc} and  Liu and Zhang \cite{lz} obtained only  the existence and multiplicity of normalized solutions for the problems they studied,
 where in the present paper we consider
  the concentration of solutions  for system \eqref{e1.1}. 
  Thus, our results fill the gap in these papers.

\item[($c$)] There are
some interesting questions worthy of further exploration. Theorem \ref{the1.1} is also valid if we replace $\Delta^{2}$ by  $\Delta^{2m}$ for $m>1,$ via similar arguments as in this paper.
On the other hand, it is natural to ask if Theorem \ref{the1.1}  remains valid if the nonlinearity satisfies $L^2$-supercritical growth?  If this is true, the remaining case $q\in (2+\frac{8}{N}, 2^{\ast\ast})$ for the result in this paper would be supplemented.
\end{itemize}

\vskip 0.5cm

\noindent
{\bf Funding.}
The first two authors were supported  by the Young Outstanding Talents Project of Scientific Innovation and Entrepreneurship of Jilin Province (No. 20240601048RC) 
and 
by the Research Foundation of Department of Education of Jilin Province (No. JJKH20251034KJ).
The third author was supported by the Slovenian Research and Innovation Agency program P1-0292 and grants J1-4031,  J1-4001, and N1-0278.

\vskip 0.5cm

\noindent
{\bf Acknowledgements.}
The authors thank the referee for all comments and suggestions.

\end{document}